\RequirePackage{fix-cm}

\documentclass{article}
\usepackage[left=3cm, right=3cm, top=2cm,bottom=3cm]{geometry}

\usepackage[utf8]{inputenc}
\usepackage[T1]{fontenc}
\usepackage{amssymb}
\usepackage{amsmath}
\usepackage{amsthm}
\usepackage{mdframed}
\usepackage{enumerate}
\usepackage{verbatim}
\usepackage{hyperref}
\usepackage{aliascnt}
\usepackage{mdwlist}
\usepackage[activate]{pdfcprot}
\usepackage{graphicx}
\usepackage{slashed}
\usepackage{mathrsfs}
\usepackage{mathtools}
\usepackage{mathdots}
\usepackage{dsfont}
\usepackage{extpfeil}
\usepackage{tabularx}	
\usepackage{multirow}
\usepackage[capitalise]{cleveref}
\usepackage{subcaption}
\usepackage{color}
\usepackage{placeins}
\usepackage[ruled,vlined]{algorithm2e}

\usepackage[all]{xy}
\usepackage{float}

\newtheorem{proposition}{Proposition}
\newtheorem*{proposition*}{Proposition}
\newtheorem{theorem}{Theorem}
\newtheorem*{theorem*}{Theorem}
\newtheorem{lemma}{Lemma}
\newtheorem*{lemma*}{Lemma}
\newtheorem{corollary}{Corollary}
\newtheorem*{corollary*}{Corollary}
\newtheorem{remark}{Remark}
\newtheorem*{remark*}{Remark}
\newtheorem{definition}{Definition}
\newtheorem*{definition*}{Definition}
\newtheorem{assumption}{Assumption}
\newtheorem*{assumption*}{Assumption}

\newcommand{\supp}{\operatorname{supp}}

\renewcommand{\setminus}{\smallsetminus}

\def\clap#1{\hbox to 0pt{\hss#1\hss}}

\newcommand{\tostar}{\overset{*}{\lower0.5em\hbox{$\smash{\scriptscriptstyle\rightharpoonup}$}}}

\DeclareMathOperator*{\argmin}{arg\,min}
\DeclareMathOperator*{\argmax}{arg\,max}

\newcolumntype{M}[1]{>{\centering\arraybackslash}m{#1}}

\AtBeginDocument{
  \label{CorrectFirstPageLabel}
  
}

\begin{document}

\title{Optimal finite element error estimates for an optimal control problem governed by the wave equation with controls of bounded variation.}
\author{S. Engel
\footnote{
Technical University of Munich, Chair for Optimal Control, Department of Mathematics,
Boltzmannstr.~3, 85748 Garching b. M\"unchen, Germany,
(\url{ 	sebastian.engel@ma.tum.de}).
Funding: The first author's work was supported by the International Research Training Group IGDK, funded by the German Science Foundation (DFG) and the Austrian Science Fund (FWF)},
P. Trautmann
\footnote{
University of Graz, Institute of Mathematics and Scientific Computing,
Heinrichstraße 36, 8010 Graz, Austria
(\url{philip.trautmann@uni-graz.at})},
B. Vexler
\footnote{
Technical University of Munich, Chair for Optimal Control, Department of Mathematics,
Boltzmannstr.~3, 85748 Garching b. M\"unchen, Germany,
(\url{vexler@ma.tum.de})},
}
\date{\today}
\maketitle


\begin{abstract}
{This work discusses the finite element discretization of an optimal control problem for the linear wave equation with time-dependent controls of bounded variation. The main focus lies on the convergence analysis of the discretization method. The state equation is discretized by a space-time finite element method. The controls are not discretized. Under suitable assumptions optimal convergence rates for the error in the state  and control variable are proven. Based on a conditional gradient method the solution of the semi-discretized optimal control problem is computed. The theoretical convergence rates are confirmed in a numerical example.}
{BV-Functions; optimal control of a wave equation; error bounds; finite elements.}
\end{abstract}



\section{Introduction}
In this paper we derive a priori error estimates for a finite element discretization of the following optimal control problem governed by the linear wave equation:
\begin{align*}
(P)\left\lbrace \begin{matrix}
& \min_{u\in BV(0,T)^m} \frac{1}{2}\| y_u-y_d \|_{L^2(\Omega_T)}^2 + \sum\nolimits_{j=1}^m \alpha_j \|D_tu_j\|_{M(I)}=:J(y,u)\\ \\
\text{subject to } & (\mathcal{W}) \left\lbrace\begin{matrix}
\partial_{tt} y - \Delta y = f = \sum\nolimits_{j=1}^m u_jg_j  &\text{ in }&I \times \Omega \\
y = 0  &\text{ on }&I\times \partial \Omega \\
(y,\partial_t y) =   (y_0 , y_1)  &\text{ in }& \lbrace 0 \rbrace \times \Omega,
\end{matrix}\right.
\end{matrix}\right.
\end{align*}
where $\Omega \subset \mathbb{R}^n$, with $n\in \lbrace 1,2,3\rbrace$, is a convex, polygonal/polyhedral bounded domain. For $T\in (0,\infty)$ we denote $I=(0,T)$. The desired state $y_d$ is assumed to satisfy $y_d\in C^1(\overline I;L^2(\Omega))$. The time depending controls $u$ are given by $u=(u_1,\cdots,u_m)\in BV(0,T)^m$, and $BV(0,T)^m$ is endowed with the norm $\|u\|_{BV(I)^m}=\sum\nolimits_{j=1}^m (\|u_j\|_{L^1(I)^m}+\|D_tu_j\|_{M(I)})$. Here $M(I)$ is the space of Borel measures, endowed with the total variation norm $\|\cdot\|_{M(I)}$. Further, let $(g_j)_{j=1}^m \subset L^\infty(\Omega)\setminus \lbrace 0 \rbrace$ with pairwise disjoint supports and $\alpha_j>0$. The initial data is chosen as $(y_0,y_1)\in H^1_0(\Omega)\times L^2(\Omega)$. Finally, we set $\Omega_T:=I\times \Omega$.

In this work we focus on controls of bounded variation in time.
By using the total variation norm in $(P)$, sparsity in the derivative of the controls is promoted, resulting in locally constant controls. This is in particular the case if the derivative of the optimal control is a linear combination of Dirac functions. Optimal control problems with $BV$-controls are already analyzed for elliptic and parabolic state equations in \cite{[CK17E],[KCK],[HMNV19],kunisch1,kunisch2,kunisch3}.

Since our article deals with a priori error estimates of a finite element discretization for the control problem $(P)$, we briefly discuss previous works on error estimates for PDE control problems with $BV$-controls. In \cite{[KCK]} the authors discretize the time-dependent $BV-$controls by cellwise constant functions. 
The state equation is discretized by piecewise constant finite elements in time and linear continuous finite elements in space. Based on this discretization approach, the authors show that the optimal value of the cost functional and the states converge with an order of $\sqrt{\tau}$ in time and linear in space. However, numerical experiments in \cite{[KCK]} indicate better results. In \cite{[HMNV19]} the authors analyze a finite element discretization of an elliptic control problem with $BV$-controls in a one dimensional setting. As in our case the controls are not discretized.
The main contribution of this work is the derivation of optimal error estimates for the control variable in the $L^1$-norm. Their analysis relies on the one dimensional setting and on structural assumption on the optimal adjoint state which guarantee that the optimal control is piecewise constant and has finitely many jumps. In our work we derive similar optimal error estimates also for the problem with a multi-dimensional wave equation and our analysis relies partially on techniques developed in the former work.

Next we briefly address the difficulties in the derivation of finite element error estimates for optimal control problems with PDEs and $BV$-controls. Standard techniques for the derivation of finite element error estimates, see e.g. \cite{Casas2012}, cannot be applied due to the non-smoothness of the cost functional and the non-reflexivity of $BV(I)$. In the last years several papers concerning the derivation of finite element error estimates for optimal control problems with measure-valued controls appeared, see e.g. \cite{[VexPie],[TVZ18]}. Using the fact that for one dimensional controls, $BV(I)$ is isomorphic to $M(I)\times \mathbb R$, some techniques from these works are used to derive error estimates for $BV$-controls.
Finally, we mention that the literature on finite element error estimates for optimal control problems governed by the wave equation is very limited. To our knowledge the only existing work in this context is \cite{[TVZ18]} which uses the space-time finite element discretization developed and analyzed in \cite{[Zlot]}. Our work also relies on this discretization method for the state equation and its error analysis.\\
The main contribution of this work is the derivation of an optimal error estimate of the control variable in the $L^1(I)$-norm and of the state variable in the $L^2(\Omega_T)$-norm. The state equation is discretized by a space-time finite element method with piecewise linear and continuous Ansatz- and test-functions from \cite{[Zlot]}. The weak formulation of the discrete state equation is augmented with a stabilization term involving the stabilization parameter $\sigma$. Stability of the method depends on the value of this parameter. Moreover, for certain values of this parameter the method is equivalent to wellknown time stepping schemes, like the Crank-Nicolson scheme or the Leap-Frog scheme. The $BV$-controls are not discretized. Due to fact that the controls are only time-dependent the problem under consideration can be reformulated as a measure-valued control problem. Based on the optimality conditions of the continuous and discrete optimal control problem the error in the state variable in the $L^2(\Omega_T)$-norm can be represented in terms of the finite element error of the state and adjoint state equation in the $L^2(\Omega_T)$-norm resp. the $L^\infty(I;L^2(\Omega))$ as well as the error in the control variable in the $L^1(I)$-norm. The convergence rates for the finite element error of the state and adjoint state are obtained from \cite{[Zlot]}. Under the assumption that the continuous and time depending functions
\[
\bar p_{1,i}:t\mapsto-\int\nolimits^T_t\int\nolimits_\Omega \overline pg_i ~dx~ ds, \text{ for }i=1,\cdots,m,
\]
where $\overline p$ is the optimal adjoint state, $\bar p_{1,i}$ is bounded by $\pm \alpha_i$, is equal to $\alpha_i$ at finitely many points in $I$, and the second derivatives of $\bar p_{1,i}$ do not vanish in these points (see (A1) and (A2)), it follows that the continuous optimal control is piecewise constant and has finitely many jumps. To obtain this information about the form of the optimal BV control, using $\bar p_{1,i}$, is particularly elementary because we consider controls in one dimension. Furthermore, it is proven that the solution of the discrete problem has the same number of jumps which are located close to the jumps of the continuous optimal control. Using these properties the error of the optimal control in the $L^1(I)$-norm is estimated in terms of the error of the state variable in the $L^2(\Omega_T)$-norm. Using a bootstrapping argument optimal rates for the error in the state and control variable as well as for the optimal value of the cost are proven. These rates are confirmed by a numerical example with known solution.\\
This work has the following structure. Section \ref{sec:wave} summarizes several needed results on the regularity of weak solutions of the wave equation. In section \ref{sec:approxwave} the space-time finite element method from \cite{[Zlot]} is presented. Moreover, important stability results as well as a priori error estimates are stated. Section \ref{sec:controlprob} deals with the reformulation of the $BV$-control problem as a measure-valued control problem and with the analysis of this problem. In particular, first order optimality conditions are derived. The next section \ref{sec:varprob} is concerned with discretization of the control problem. It is based on the mentioned space-time finite element method and the variational discretization concept. In section \ref{sec:errorest} the error estimates for the optimal state and control variable as well as the optimal functional value are derived. Finally, in section \ref{sec:numerics} a generalized conditional gradient method is introduced which applicable in the context of controls which are not discretized. Based on this method a problem with known solution is solved and the theoretical error estimates are confirmed.

\section{Preliminaries on the Wave Equation}\label{sec:wave}
We consider $\Omega\subset \mathbb R^{d}$, $d=1,2,3$ as convex, polygonal/polyhedral domain. Let $\lbrace \lambda_k \rbrace_{k\in\mathbb{N}}$ be the non-decreasing eigenvalues of the Laplace operator $-\Delta$ with homogeneous boundary conditions and let $\lbrace \mu_k \rbrace_{k\in \mathbb{N}}$ be the corresponding system of eigenfunctions, which are orthonormal complete in $L^2(\Omega)$, and orthogonal complete in $H^1_0(\Omega)$. Hence, let us introduce for $\alpha\geq 0$
the Hilbert spaces
\begin{equation*}
\mathbb{H}^\alpha=\left\lbrace w\in L^2(\Omega) \left\vert \|w\|_{\mathbb H^{\alpha}}^2:=\sum\nolimits_{k\geq 1} \left(\lambda^{(k)} \right)^\alpha \left\langle w,\mu_k\right\rangle_{L^2(\Omega)}^2<\infty
\right. \right\rbrace.
\end{equation*}
For $\alpha=0,1$ we get $L^2(\Omega)$ respectively $H^1_0(\Omega)$. The convexity of $\Omega$ implies that $\mathbb{H}^2= H^2(\Omega)\cap H^1_0(\Omega)$. In general holds $\mathbb{H}^\alpha \hookrightarrow \mathbb{H}^\beta$ for $\alpha \geq \beta$. We denote the dual space of $\mathbb H^\alpha$ by $\mathbb H^{-\alpha}$.
Next we introduce the weak solution of the wave equation with the forcing function $f$, initial displacement $y_0$, and initial velocity $y_1$.
\begin{definition}(\cite{[La68]}, Chap.IV, Sec.4)\\
Let $(f,y_0,y_1)\in L^1(I;L^2(\Omega))\times H^1_0(\Omega)\times L^2(\Omega)$. We call a function $y\in C(\overline I;H^1_0(\Omega))$ with $\partial_t y \in C(\overline{I};L^2(\Omega))$ a weak solution of $(\mathcal{W})$ , if
\begin{equation}\label{weakSoluVariForm}
	\int_0^T-(\partial_t  y ,\partial_t \eta)_{L^2(\Omega)}+(\nabla y,\nabla\eta )_{L^2(\Omega)}~dt=(y_1,\eta(0) )_{L^2(\Omega)}+\int_0^T(f,\eta)_{L^2(\Omega)}~dt
\end{equation}
for any $\eta\in L^1(I;H^1_0(\Omega))$ such that $\partial_t \eta \in L^1(I;L^2(\Omega))$, $\eta\vert_{t=T}=0$, and $y$ satisfies the initial condition $y\vert_{t=0}=y_0$. \label{weakSoluti}
\end{definition}
For the following existence and regularity results of weak solutions of the wave equation we refer to \cite[Proposition 1.1., 1.3.]{[Zlot]}:
\begin{theorem}\label{energy}
For each $(f,y_0,y_1)\in L^1(0,T;L^2(\Omega))\times H^1_0(\Omega) \times L^2(\Omega)$ there exists a unique weak solution $y$ of $(\mathcal{W})$. Moreover, there exists a constant $c>0$ such that the weak solution $y$ satisfies
\begin{equation}\label{energy2}
\|y\|_{C(\overline{I};\mathbb H^{\alpha+1})} + \|\partial_t y\|_{C(\overline{I};\mathbb H^\alpha)}+\|\partial_{tt}y\|_{L^\kappa(I;\mathbb H^{\alpha-1})} \leq c\left(\|y_0\|_{\mathbb H^{\alpha+1}}  +\|y_1\|_{\mathbb H^{\alpha}}+ \|f\|_{L^\kappa(I;\mathbb H^{\alpha})}\right)
\end{equation}
provided $(f,y_0,y_1)\in L^\kappa(I;\mathbb H^{\alpha}) \times \mathbb H^{\alpha+1} \times \mathbb H^{\alpha}$ and
\begin{equation}\label{energy3}
\|y\|_{C(\overline{I};\mathbb H^{\alpha +2})} +\|\partial_t y\|_{C(\overline{I};\mathbb H^{\alpha+1})} + \|\partial_{tt} y\|_{C(\overline{I};\mathbb H^\alpha)}\\ \\
\leq c\left(\|y_0\|_{\mathbb H^{\alpha +2}}+\|y_1\|_{\mathbb H^{\alpha+1}}+\|f\|_{W^{1,1}(I;\mathbb H^{\alpha})}\right).
\end{equation}
provided $(f,y_0,y_1)\in W^{1,1}(I;\mathbb H^{\alpha}) \times \mathbb H^{\alpha+2} \times \mathbb H^{\alpha+1}$ with $0\leq \alpha$, $1\leq \kappa\leq\infty$.
\end{theorem}
\begin{proof}
The proof can be found in \cite[Proposition 1.3, Remark 1.2]{[Zlot]}.
\end{proof}
\begin{definition}
Let us define the following continuous linear operators:
\begin{equation*}
L\colon L^2(\Omega_T) \rightarrow  L^2(\Omega_T),~f \mapsto y(f)\quad \text{and}\quad Q\colon H^1_0(\Omega)\times L^2(\Omega) \rightarrow L^2(\Omega_T),~(y_0,y_1) \mapsto y(y_0,y_1)
\end{equation*}
The function $y(f)$ denotes the weak solution of the wave equation with $y_0=y_1=0$ and forcing function $f$. The function
$y(y_0,y_1) $ denotes the weak solution of the wave equation with initial datum $y_0$ and $y_1$ and $f=0$.
\end{definition}
\begin{lemma}\label{adjointLemma}
The adjoint operator $L^*:L^2(\Omega_T)\rightarrow L^2(\Omega_T)$ of $L$ is given by $w \mapsto p(w)$ where $p(w)\in C(\overline I;H^1_0(\Omega))\cap C(\overline I;L^2(\Omega))$ is the weak solution of the backwards in time equation
\begin{align}
(\mathcal{W}^*) \left\lbrace\begin{matrix}
	\partial_{tt} p - \Delta p = w &\text{ in }&I \times \Omega \\
	p = 0  &\text{ on }&I\times \partial \Omega \\
	(p,\partial_t p) =   (0 , 0)  &\text{ in }& \lbrace T \rbrace \times \Omega.
	\end{matrix}\right.
\end{align}
\end{lemma}
\section{Approximation of the Wave Equation}\label{sec:approxwave}
In the following we introduce the space-time finite element method for the discretization of the wave equation. This method can be found in \cite{[Zlot]}.
We consider a mesh $\mathcal{T}_h$ consisting of a finite set of triangles (for $d=2$) or tetrahedra (for $d=3$) $K$ with $h=\max_{K\in\mathcal{T}_h}\rho(K)$, where $\rho(K)$ denotes the diameter of $K$.
We assume that the family of meshes $(\mathcal{T}_h)_h$ is admissible, shape regular and  quasi-uniform.
Since $\Omega$ is polygonal and convex, we require that $\Omega=\bigcup_{K\in \mathcal T_h}K$ holds. We denote the space of piecewise linear and continuous finite elements based on the triangulation $\mathcal{T}_h$ by $S_h\subset H^1_0(\Omega)\cap C(\overline \Omega)$ and its nodal basis by$(\varphi_i)_{i=1}^N$.
\subsection{Space-Time Finite Element Method}\label{3levelMt}
We discretize the time interval $I$ uniformly with the time nodes $0=t_0 < ... < t_M=T$ and the stepsize $\tau=T/M$. We denote the set of time nodes by $\overline{w}^\tau=\lbrace t_0,...,t_M\rbrace$.
Then we introduce the space of piecewise linear and continuous functions with respect to $\overline{w}^\tau$ by
\begin{equation*}
S_\tau:=\left\lbrace w\in C(\overline I)\vert\quad w|_{[t_{k-1},t_k]} \text{ linear },\textbf{ } 1\leq k \leq M  \right\rbrace.
\end{equation*}
The standard hat functions form a basis $e_m(t_k)=\delta_{mk}$, $m,k=0,\ldots,M$ of this discrete space. Finally, we use the notation $\vartheta:=(\tau,h)$ with $\tau,h>0$.
\begin{definition}\label{discreteSolutdef}
Let $\sigma\geq 0$. We call $y_{\vartheta}\in \hat{S}_{\vartheta}:=\text{span}\lbrace v_h\cdot v_\tau\vert v_h\in S_h\text{, }v_\tau\in S_\tau \rbrace$ a discrete solution of (\ref{weakSoluVariForm}) if $y_{\vartheta}$ satisfies:
\begin{multline}\label{discreteSolutdefeq}
\int_0^T-(\partial_t y_{\vartheta},\partial_t\eta)_{L^2(\Omega)}-(\sigma -\frac{1}{6})\tau^2( \nabla \partial_t y_{\vartheta},\nabla \partial_t\eta)_{L^2(\Omega)}+ (\nabla y_{\vartheta},\nabla \eta)_{L^2(\Omega)}~dt\\
= (y_1,\eta(0))_{L^2(\Omega)} + \int_0^T(f,\eta)_{L^2(\Omega)}~dt
\end{multline}
for all $\eta \in \hat{S}_{\vartheta}$ with $\eta(T)=0$ and initial condition $y_{\vartheta}(0):=R_hy_{0}$, where $R_h$ is the Ritz projection on $S_h$, i.e.
\begin{equation*}
(\nabla R_hy_{0},\nabla\varphi)_{L^2(\Omega)}=(\nabla y_0,\nabla \varphi)_{L^2(\Omega)}\quad \forall\varphi \in S_h.
\end{equation*}
\begin{remark}
Here $\sigma$ plays the role of a stabilization parameter. With an increasing value of $\sigma$ the method becomes more stable. For $\sigma\geq1/4$ the method is unconditionally stable, see \cite{[Zlot]}.
\end{remark}
\end{definition}
\subsection{A Priori Error Estimates for the Space-Time Finite Element Method}
Next we make an assumption on the relationship between $\tau$ and $h$ which ensures stability of the method for $0\leq \sigma<1/4$.
\begin{assumption}\label{AssumpOntauh}
Let $\varepsilon_0\in (0,1]$ be arbitrary and fixed. Moreover, let $c_1$ be the smallest constant in the inverse inequality $\|\nabla \varphi\|_{L^2(\Omega)}\leq c_1 h^{-1}\|\varphi\|_{L^2(\Omega)}$ for all $\varphi_h\in S_h$. Moreover, let a $c_2$ be the constant in this a priori estimate for the Ritz projection $\|w-R_hw\|_{L^2(\Omega)}\leq c_2h\|\nabla w\|_{L^2(\Omega)}$. From now on it is assumed that
\begin{enumerate}
  \item $\sigma\geq \frac 14-\frac{c_1h^2(1-\varepsilon_0^2)}{\tau^2}$,
  \item $\sigma\geq \frac{1+\varepsilon_0^2}{4}-\frac{c_1h^2}{\tau^2}$,
  \item $\vert \sigma\vert \tau^2 \leq 2(c_2h^2 +\tau^2)$.
\end{enumerate}
\end{assumption}
\begin{remark}
This space-time finite element method is related to well-known time-stepping schemes. For $\sigma=0$ it is related to the explicit Leap-Frog-method and for $\sigma=\frac{1}{4}$ to the Crank-Nicolson scheme, see also \cite[Remark 5.1, 5.4]{vextra}.
A more detailed discussion can be found in \cite{[Zlot]}.
\end{remark}
\begin{lemma}\label{LtauhAbsch}
The solution $y_\vartheta$ of \eqref{discreteSolutdefeq} for $(f,y_0,y_1)\in L^1(I,L^2(\Omega))\times H^1_0(\Omega)\times L^2(\Omega)$ satisfies the following inequality
\begin{equation}\label{bddlinOpLtauh}
\|y_\vartheta\|_{C(\overline I;L^2(\Omega))}\leq c\left(\|y_0\|_{H^1_0(\Omega)}+\|y_1\|_{L^2(\Omega)}+\|f\|_{L^1(I;L^2(\Omega))}\right)
\end{equation}
with a constant $c$ independent of $h$, $f$, $y_0$ and $y_1$.
\end{lemma}
\begin{proof}
The result follows directly from \cite[Theorem 2.1, Remark 2.1]{[Zlot]}.
\end{proof}
\begin{theorem}\label{AprioriThm}
The following error estimates hold:
\begin{equation}\label{Cor221}
\|y-y_{\vartheta}\|_{C(\overline{I};L^2(\Omega))}\leq
c(h^2+\tau^2)\left(\|y_0\|_{\mathbb{H}^{3}}+\|y_1\|_{\mathbb{H}^{2}}+\|f\|_{L^\infty(I;\mathbb H^{2})} \right)
\end{equation}
provided $(f,y_0,y_1)\in L^\infty(I;\mathbb H^2)\times \mathbb{H}^{3}\times \mathbb{H}^{2}$ as well as
\begin{equation}\label{AprioriEsti}
\|y-y_{\vartheta}\|_{C(\overline{I};L^2(\Omega))}\leq
c(h^2+\tau^2)\left(\|y_0\|_{\mathbb{H}^{3}}+\|y_1\|_{\mathbb{H}^{2}}+\|f\|_{W^{1,1}(I;H^1_0(\Omega))} \right)
\end{equation}
provided $(f,y_0,y_1)\in W^{1,1}(I;H^1_0(\Omega))\times \mathbb{H}^{3}\times \mathbb{H}^{2}$.
\end{theorem}
\begin{proof}
The result follows directly from \cite[Theorem 4.1., 4.3. and comments in its proof]{[Zlot]}.
\end{proof}
\section{Equivalent Problem $(\tilde{P})$ }\label{sec:controlprob}
In this section we introduce a specific isomorphism between $BV(I)^m\otimes \{(g_j)_{j=1}^m\}$ and $M(I)^m\times \mathbb{R}^m$. Based on this isomorphism $(P)$ is equivalently formulated as a measure valued control problem. First of all we prove existence and uniqueness of a solution to $(P)$.
\begin{theorem}\label{ExistenceSoluti} Problem $(P)$ has a unique solution in $BV(I)^m$.
\end{theorem}
\begin{proof}
Utilizing the fact, that the forward mapping is continuous from $L^2(I)^m$ to $L^2(\Omega_T)$, the proof can be carried out along the line of \cite[Theorem 3.1]{[KCK]}.
\end{proof}
Next we introduce several linear and continuous operators and discuss its properties. The operator $B\colon  M(I)^m \times \mathbb{R}^m  \rightarrow  L^2(\Omega_T)$ is given by
\begin{equation}\label{B_red_to_L_2}
(v,c)  \mapsto  \sum\nolimits_{j=1}^m \left( \int\nolimits_0^t~dv_j(s) -\frac1T\int\nolimits_0^T\int\nolimits_0^t~dv_j(s)~ds + c_j \right)g_j.
\end{equation}
The measures $v_j$ are the derivatives of the generated BV-function and $c_j$ are the mean values.
Next, we define the predual operator of $B$ given by $B^*\colon L^2(\Omega_T)\to C_0(I)^m\times \mathbb R^m$
\[
B^*\colon q\mapsto \left(w'_1,\ldots,w'_m,\int\nolimits_0^T\int\nolimits_\Omega qg_1~dx~\mathrm dt,\ldots,\int\nolimits_0^T\int\nolimits_\Omega qg_m~ dx~dt\right)
\]
where $w\in H^2(I)$ solves
\begin{equation}\label{eq:timepde}
\left\{\begin{aligned}
-w''_j&=\int\nolimits_\Omega q(\cdot,x)g_j(x)~dx-\frac1T\int\nolimits_0^T\int\nolimits_\Omega q(t,x)g_j(x)~dx~dt\quad\text{in }(0,T)\\
w'_j(0)&=w'_j(T)=0\quad \text{with} \quad \int\nolimits_0^Tw_j(t)~dt=0\quad\text{for}\quad j=1,\ldots,m.
\end{aligned}
\right.
\end{equation}
\begin{proposition}
The operator $B^*\colon L^2(\Omega_T)\to C_0(I)^m\times \mathbb R^m$ is well defined and the predual of $B$, i.e. the following holds
\[
\int\nolimits_{\Omega_T}B(v,c)q~\mathrm dx~\mathrm dt=\langle (v,c),B^*(q)\rangle
\]
for all $(v,c)\in M(I)^m\times\mathbb R^m$ and for all $q\in L^2(\Omega_T)$.
\end{proposition}
\begin{proof}
The equation \eqref{eq:timepde} has a unique solution $w_j\in H^2(I)$, since $\int\nolimits_\Omega q(\cdot,x)g_j~\mathrm dx-\frac1T\int\nolimits_0^T\int\nolimits_\Omega q(t,x)~\mathrm dx~\mathrm dt\in L^2(I)$ and has zero mean. Moreover, we have $w'_j\in H^1_0(I)\hookrightarrow C_0(I)$. Thus, the operator $B^*$ is well defined.
Moreover, there holds
\begin{multline*}
\langle (v,c),B^*(q)\rangle=\sum\nolimits_{j=1}^m\int\nolimits_0^Tw'_j~ dv_j+\sum\nolimits_{j=1}^mc_j\int\nolimits_0^T\int\nolimits_\Omega qg_j~ dx~ dt\\
=\sum\nolimits_{j=1}^m\int\nolimits_0^T-w''_j\int\nolimits_0^t~ dv_j~ dt+\sum\nolimits_{j=1}^mc_j\int\nolimits_0^T\int\nolimits_\Omega qg_j~ dx~ dt\\
=\sum\nolimits_{j=1}^m\int\nolimits_0^T\left(\int\nolimits_\Omega qg_j~ dx-\frac1T\int\nolimits_0^T\int\nolimits_\Omega qg_j~ dx~ dt\right)\int\nolimits_0^t~ dv_j~ dt+\sum\nolimits_{j=1}^mc_j\int\nolimits_0^T\int\nolimits_\Omega qg_j~ dx~ dt\\
=\int\nolimits_0^T\int\nolimits_\Omega q\sum\nolimits_{j=1}^m\left(\int\nolimits_0^t~ dv_j-\frac1T\int\nolimits_0^T\int\nolimits_0^t~ dv_j~ dt+c_j\right)g_j~ dx~ dt
=\int\nolimits_{\Omega_T}B(v,c)q~ dx~ dt
\end{multline*}
for all $(v,c)\in M(I)^m\times\mathbb R^m$ and for all $q\in L^2(\Omega_T)$. The use of integration by parts is justified by the density of $C_c^\infty(I)$ in $H^1_0(I)$.
\end{proof}
\begin{proposition}\label{Propdiff}
Let $w_j\in H^2(I)$, $j=1,\cdots,m$ be the solution of \eqref{eq:timepde}. Then there holds
\[
w'_j(t)=\int\nolimits_t^T\int\nolimits_\Omega q(s,x)g_j(x)~ dx~ ds+\frac{(t-T)}T\int\nolimits_0^T\int\nolimits_\Omega q(t,x)g_j(x)~ dx~ dt.
\]
\end{proposition}
\begin{proposition}
The operator $B\colon M(I)^m\times \mathbb R^m\to BV(I)^m\otimes \{(g_j)_{j=1}^m\}$ is an isomorphism.
\end{proposition}
\begin{proof}
The inverse of $B$ is given by
\[
B^{-1}\colon \sum\nolimits_{j=1}^mu_jg_j\mapsto \left(u'_1,\ldots,u'_m,\frac 1T\int\nolimits_0^Tu_1~ dt,\ldots, \frac 1T\int\nolimits_0^T u_m~ dt \right).
\]
\end{proof}
Using $B$ we can rewrite $(P)$ into the equivalent problem
\begin{align*}
(\tilde{P})\left\lbrace\begin{matrix}
& \min_{
\begin{matrix}
v\in M(I)^m \\
c\in \mathbb{R}^m
\end{matrix}} \frac{1}{2}\| S(v,c)-y_d \|_{L^2(\Omega_T)}^2 + \sum\nolimits_{j=1}^m \alpha_j \|v_j\|_{M(I)}=:\tilde{J}(v,c),
\end{matrix}\right.
\end{align*}
with $S\colon  M(I)^m \times \mathbb{R}^m  \to  L^2(\Omega_T)$ defined by $(v,c) \mapsto L(B(v,c))+Q(y_0,y_1)$.
\subsection{First-Order optimality condition of $(\tilde{P})$ }
In the following a necessary and sufficient first-order optimality condition of $(\tilde{P})$ is presented as well as sparsity results for the derivative of the optimal control.
Let $(\overline v,\overline c)$ be the unique optimal pair. We define the quantities $\overline p=L^*(S(\overline v,\overline c)-y_d)$ and $\overline p_1\in C(\overline I)^m$ by
\[
\overline p_{1,i}:=-\int\nolimits^T_t\int\nolimits_\Omega \overline pg_i ~dx~ ds
\]
for $i=1,\ldots,m$.
\begin{theorem}\label{FirstOrder}
The pair $(\overline v,\overline c)\in M(I)^m\times\mathbb{R}^m$ is an optimal control of $(\tilde{P})$ if and only if
\begin{align}\label{p_1_i_function}
\overline p_{1,i}&\in \alpha_i \partial \|\overline v_i\|_{M(I)}\quad i=1,\ldots, m,\\
\overline p_{1}(0)&=0.
\end{align}
Equivalently it holds
\begin{equation}\label{equiv}
	\langle v-\overline v_i,\overline{p}_{1,i}\rangle_{M(I),C_0(I)}+ \alpha_i \|\overline v_i\|_{M(I)} \leq \alpha_i \|v\|_{M(I)} \quad \forall v\in M(I)~\text{and}~i=1,\ldots,m
\end{equation}	
and $\overline p_1(0)=0$.
\end{theorem}
\begin{proof}
The proof of Theorem \ref{FirstOrder} is done along the lines of the proof of\\
\cite[Theorem 3.3]{[KCK]}.
By the convexity of $(\tilde{P})$ we have, that $(\overline v,\overline c)\in M(I)^m\times\mathbb{R}^m$ is an optimal control of $(\tilde{P})$ if and only if
\begin{equation*}
0\in \partial \left(
\frac{1}{2}\| S(\overline v,\overline c)-y_d \|_{L^2(\Omega_T)}^2 + \sum\nolimits_{j=1}^m \alpha_j \|\overline v_j\|_{M(I)}
\right)\subseteq (M(I)^m\times \mathbb{R}^m)^*=(M(I)^\ast)^m\times \mathbb R^m.
\end{equation*}
Define the following function $F(v,c):=\frac{1}{2}\| S(v,c)-y_d \|_{L^2(\Omega_T)}^2$ for $(v,c)\in M(0,T)^m\times \mathbb{R}^m$.
Its Gateaux derivative has the form
\[
DF_{(v,c)}(v,c)=B^\ast L^\ast(S(v,c)-y_d)\in C_0(I)^m\times \mathbb R^m
\]
According to the theory of convex analysis, e.g. \cite[Proposition 5.6]{NewEkTe}, we have
\begin{equation}\label{CCKTrick2}
0\in
DF_{(v,c)}(\overline v,\overline c) + \partial\left(\sum\nolimits_{i=1}^m \alpha_i \|\overline v_i\|_{M(I)}
\right).
\end{equation}
Using
\[
\partial\left(\sum\nolimits_{i=1}^m \alpha_i \|\overline v_i\|_{M(I)}
\right)=\begin{pmatrix}
			\left( \alpha_i \partial \| \overline v_i \|_{M(I)}\right)_{i=1}^m \\
			0
			\end{pmatrix}
\subseteq (M(I)^\ast)^m\times \mathbb R^m
\]
and (\ref{CCKTrick2}) as well as Proposition \ref{Propdiff} imply
\begin{equation}
\overline p_{1,i}\in \alpha_i \partial \| \overline v_i \|_{M(I)}\quad\forall i=1,\ldots,m,\quad \overline p_1(0)=0.
\end{equation}
\end{proof}
The following proposition is a consequence of \cite[Proposition 3.2.]{CK14SemiLinEllip}:
\begin{proposition}\label{support_measure_cor}
Let $(\overline v,\overline c)\in M(I)^m\times\mathbb{R}^m$ be an optimal control of $(\tilde{P})$, then for all $i=1,\cdots,m$ and $\overline{p}_{1}=(\overline p_{1,i})_{i=1}^m$ given in (\ref{p_1_i_function}) holds
\begin{itemize}
\item[a)]$\|\overline p_{1,i}\|_{C_{0}(I)}\leq \alpha_i$,
\item[b)]$\int\nolimits_0^T -\frac{\overline p_{1,i}}{\alpha_i}d\overline{v}_i=\int\nolimits_0^T d\vert \overline{v}_i\vert=\|\overline{v}_i\|_{M(I)}$,
\item[c)]$\supp(\overline{v}_i^{\pm})\subseteq \lbrace t \in I \vert \overline p_{1,i}(t)=\pm \alpha_i\rbrace$,
where $\overline v_i=\overline v_i^+-\overline v_i^-$ is the Jordan decomposition of $\overline v_i$.
\end{itemize}
\end{proposition}
\begin{remark}\label{fsa}
Let us note that the boundary property of $\overline p_1$, i.e. $\overline p_1(0)=\overline p_1(T)=0$, and the continuity of $ \overline p_1$ imply with Proposition \ref{support_measure_cor}, c), that there exists a $\epsilon_i>0$ such that $dist(\supp(\overline{v}_i^{\pm}),\lbrace 0, T\rbrace)>\epsilon_i.$
\end{remark}
\section{The Variationally Discretized Problem}\label{sec:varprob}
In this section we introduce a discretized version of $(\tilde{P})$ and discuss its properties. We use the concept of variational discretization in which the control is not discretized. In particular, we consider the problem $(\tilde{P}_{\vartheta}^{\text{semi}})$:
\begin{align*}
(\tilde{P}_{\vartheta}^{\text{semi}})\left\lbrace\begin{matrix}
& \min_{
	\begin{matrix}
	v\in M(I)^m \\
	c\in \mathbb{R}^m
	\end{matrix}}
\frac{1}{2}\| S_{\vartheta}(v,c)-y_d \|_{L^2(\Omega_T)}^2 + \sum\nolimits_{j=1}^m \alpha_j \|v_j\|_{M(I)}=:J_{\vartheta}(v,c)
\end{matrix}\right.
\end{align*}
with $S_{\vartheta}\colon  M(I)^m \times \mathbb{R}^m  \longrightarrow  L^2(\Omega_T)$ defined by
$
(v,c) \mapsto  L_{\vartheta}(B(v,c))+Q_{\vartheta}(y_0,y_1).
$
Here $L_{\vartheta}\colon  L^2(\Omega_T)  \rightarrow   L^2(\Omega_T)$ is defined by $f\mapsto y_\vartheta(f)$, where $y_\vartheta(f)$ solves \eqref{discreteSolutdefeq} for a source f and $(y_0,y_1)=(0,0)$. The operator $Q_{\vartheta}\colon  H^1_0(\Omega)\times L^2(\Omega)  \rightarrow   L^2(\Omega_T)$ is defined by $(y_0,y_1)\mapsto y_\vartheta(y_0,y_1)$, where $y_\vartheta(y_0,y_1)$ solves \eqref{discreteSolutdefeq} with $(y_0,y_1)$ as initial datum and $f=0$.
\begin{remark}\label{AdjDiscOp}
We can represent the adjoint of $L_{\vartheta}$ in the form  $w\mapsto L_{\vartheta}^*(w)(t,x)=L_{\vartheta}(w\circ\widetilde{\phi})(\widetilde{\phi}(t,x))$ with $\widetilde{\phi}(t,x)=(T-t,x)$, and $w\in L^2(\Omega_T)$. This is true since $L_{\vartheta}(f)(0)=0$ and $L_{\vartheta}^*(w)(T)=0$ and thus $L_{\vartheta}(f)$ and $L_{\vartheta}^*(w)$ can be used in \eqref{discreteSolutdefeq} as test functions for the forwards and backwards equation.  Hence, Theorem \ref{AprioriEsti} and Lemma \ref{LtauhAbsch} are valid for $L_{\vartheta}^*(w)$ as well.
\end{remark}
\begin{theorem}
The problem $(\tilde{P}_{\vartheta}^{\text{semi}})$ has a solution in $M(I)^m \times \mathbb{R}^m$.
\end{theorem}
\begin{proof}
The existence of an optimal control for $(\tilde{P}_{\vartheta}^{\text{semi}})$ can be similarly shown as in the proof of Theorem \ref{ExistenceSoluti}.
\end{proof}
Note, that a BV-representation of the solutions $(\overline v, \overline c)$, $(\overline v_\vartheta,\overline c_{\vartheta})$ of $(\tilde{P})$, respectively $(\tilde{P}_{\vartheta}^{\text{semi}})$ are defined by
\begin{equation}\label{repBV}
\overline u(t):=\int\nolimits_0^t d\overline v(s)-\frac1T\int\nolimits_0^T\int\nolimits_0^t d\overline v(s)~dt+\overline c,\quad \text{and}\quad
\overline u_{\vartheta}(t):=\int\nolimits_0^t d\overline v_{\vartheta}(s)-\frac1T\int\nolimits_0^T\int\nolimits_0^t d\overline v_{\vartheta}(s)~dt+\overline c_{\vartheta}.
\end{equation}
Next we define the quantities $\overline p_\vartheta=L^*_\vartheta(S_\vartheta(\overline v_\vartheta,\overline c_\vartheta)-y_d)$ and
\[
\overline p_{1,\vartheta,j}:=-\int\nolimits^T_t\int\nolimits_\Omega \overline p_\vartheta g_j ~dx~ ds\quad \text{for}\quad j=1,\ldots,m,
\]
which is continuously differentiable and piecewise quadratic in time.
\begin{theorem}\label{subDifTheo}
The pair $(\overline v_{\vartheta},\overline c_{\vartheta})\in M(I)^m\times\mathbb{R}^m$ is a optimal control of $(\tilde{P}_{\vartheta}^{\text{semi}})$ if and only if
\begin{align}\label{FirstOrderThm}
\overline p_{1,\vartheta,i}&\in \alpha_i \partial \|\overline v_{\vartheta,i}\|_{M(I)}\quad i=1,\ldots, m,\\
\overline p_{1,\vartheta}(0)&=0.
\end{align}
Equivalently it holds
\begin{equation}\label{equiv2}
	\langle v-\overline v_i,\overline p_{1,\vartheta,i}\rangle_{M(I),C_0(I)}+\alpha_i \|\overline v_i\|_{M(I)} \leqslant \alpha_i \|v\|_{M(I)} \quad \forall v\in M(I),~i=1,\ldots,m,
\end{equation}	
and $\overline p_{1,\vartheta}(0)=0$.
\end{theorem}
\begin{proof} The proof is similar to Theorem \ref{FirstOrder}.\end{proof}
\begin{remark}\label{RauteRemark}
Due to Theorem \ref{subDifTheo}, we can show that Proposition \ref{support_measure_cor} holds analogiously for $(\tilde{P}^{semi}_\vartheta)$.
\end{remark}
\section{A Priori Error Estimates}\label{sec:errorest}
In this section error estimates of problem $(\tilde{P}_{\vartheta}^{\text{semi}})$ for the optimal control, optimal state and optimal cost functional value are presented. Under specific assumptions, we proof optimal rates for the optimal control, state and cost.
For reason of convenience, the following notation is introduced.
For an optimal control $(\overline{v}_{\vartheta},\overline{c}_{\vartheta}) \in  M(I)^m\times\mathbb{R}^m$ of $(\widetilde{P}_{\vartheta}^{\text{semi}})$ and the optimal control $(\overline{v},\overline{c})\in M(I)^m\times\mathbb{R}^m$ of $(\widetilde{P})$ we introduce the corresponding optimal states by
$
\overline{y}_{\vartheta}:=S_{\vartheta}(\overline{v}_{\vartheta},\overline{c}_{\vartheta})
$ and
$
\overline{y}:=S(\overline{v},\overline{c}).
$
Further, we define the mixed state by $\hat{y}_{\vartheta}:=L_{\vartheta}(B(\overline{v},\overline{c}))+Q_{\vartheta}(y_0,y_1)$.
The mixed adjoint state is chosen as $\hat{p}_{\vartheta}:=L_{\vartheta}^*(\overline{y}-y_d)$.
In the proofs of following the Lemmata and  Theorem, we use similar steps as in the proof of \cite[Theorem 4.4]{[VexPie]}.
\begin{lemma}\label{lemmaImply1}
There holds
\begin{equation}\label{starIneqRate}
\left\langle \overline p_{1,\vartheta}- \overline p_{1},\overline{v}_{\vartheta}-\overline{v} \right\rangle\leq 0
\end{equation}
with $(\overline{v},\overline{c})$ as the optimal control of $(\tilde{P})$ and $(\overline{v}_{\vartheta},\overline{c}_{\vartheta})$ as a solution of $(\tilde{P}_{\vartheta}^{\text{semi}})$.
\end{lemma}
\begin{proof}
Inequality (\ref{starIneqRate}) follows from monotonicity of the subdifferential.
\end{proof}
\begin{lemma}\label{firstEqError}
Consider optimal control $(\overline v,\overline c)$ of $(\tilde P)$,  and $(\overline v_\vartheta,\overline c_\vartheta)$ of $(\tilde{P}^{semi}_{\vartheta})$, as well as their BV-representations $\overline u$, and $\overline u_\vartheta$. For the optimal states $\overline{y}$ and $\overline{y}_{\vartheta}$ of problem $(\widetilde{P})$, respectively $(\tilde{P}^{semi}_{\vartheta})$, we have
\begin{equation}\label{ineq1}
\|\overline{y}_{\vartheta}-\overline{y}\|_{L^2(\Omega_T)} \leq c\|\overline{y}-\hat{y}_{\vartheta}\|_{L^2(\Omega_T)}+c\|\overline{u}_{\vartheta}-\overline{u}\|_{L^1(I)^m}^{\frac{1}{2}}\|\overline{p} -   \hat{p}_{\vartheta}\|_{L^\infty(I;L^2(\Omega))}^\frac{1}{2}
\end{equation}
with a constant $c>0$ depending on $g$.
\end{lemma}
\begin{proof}
Lemma \ref{lemmaImply1}, the  properties of $B$ and $B^*$ and the fact that $\overline p_1(0)=\overline p_{1,\vartheta}(0)=0$ imply the following
\begin{multline*}
0\geq
\left\langle p_{1,\vartheta}-p_1,\overline{v}_{\vartheta}-\overline{v}\right\rangle=\left\langle B^*(\overline{p}_{\vartheta}-\overline{p}),(\overline{v}_{\vartheta}-\overline{v},\overline{c}_{\vartheta}-\overline{c})\right\rangle
= \left(\overline{p}_{\vartheta}-\hat{p}_\vartheta,(\overline{u}_\vartheta-\overline{u})\cdot g\right)_{L^2(\Omega_T)}\\
+\left(\hat{p}_\vartheta-\overline{p},(\overline{u}_\vartheta-\overline{u})\cdot g\right)_{L^2(\Omega_T)}\\
=\left(\overline{y}_{\vartheta}-\overline{y},L_\vartheta((\overline{u}_\vartheta-\overline{u})g)\right)_{L^2(\Omega_T)}+\left(\hat{p}_\vartheta-\overline{p},(\overline{u}_\vartheta-\overline{u})\cdot g\right)_{L^2(\Omega_T)}\\
=\|\overline{y}_\vartheta-\overline{y}\|_{L^2(\Omega_T)}^2+\left\langle\overline{y}_\vartheta-\overline{y},\overline{y}-\hat{y}_\vartheta\right\rangle_{L^2(\Omega_T)}+\left\langle\hat{p}_\vartheta-\overline{p},(
\overline{u}_\vartheta-\overline{u})\cdot g\right\rangle_{L^2(\Omega_T)}.
\end{multline*}
From these calculations we obtained (\ref{ineq1}) by
\begin{multline*}
\|\overline{y}_{\vartheta}-\overline{y}\|_{L^2(\Omega_T)}^2 \leq \left(\overline{y}-\overline{y}_{\vartheta},\overline{y}-\hat{y}_{\vartheta}\right)_{L^2(\Omega_T)}+\left(\overline{p} - \hat{p}_{\vartheta},(\overline{u}_{\vartheta}-\overline{u})\cdot g \right)_{L^2(\Omega_T)}\\
\leq\frac{1}{2}\|\overline{y}-\overline{y}_{\vartheta}\|_{L^2(\Omega_T)}^2+\frac{1}{2}\|\overline{y}-\hat{y}_{\vartheta}\|_{L^2(\Omega_T)}^2+c\|\overline{u}_{\vartheta}-\overline{u}\|_{L^1(I)^m}\|\overline{p} - \hat{p}_{\vartheta}\|_{L^\infty(I;L^2(\Omega))}.
\end{multline*}
\end{proof}
\begin{lemma}\label{boundedDiscreteoptiInBVandLInfty}
The sequence of the BV representatives $( \overline{u}_{\vartheta})_{\vartheta}$ of the optimal controls of $(\tilde{P}_{\vartheta}^{\text{semi}})$ are bounded in $BV(I)^m$ with respect to $\vartheta\rightarrow 0$.
\end{lemma}
\begin{proof}
At first, we show that
\[\overline{u}_{\vartheta}=\int\nolimits_0^t d\overline{v}_{\vartheta}(s)-\frac1T \int\nolimits_0^T\int\nolimits_0^t d\overline{v}_{\vartheta}(s)~ds+\overline{c}_{\vartheta}=\hat{u}_{\vartheta}+\overline{c}_{\vartheta}\]
is bounded in $BV(I)^m$ for $\vartheta\rightarrow 0$. Due to the optimality of $\overline{u}_\vartheta$, holds the inequality $J_\vartheta(\overline{u}_{\vartheta})\leq J_\vartheta(0)$ for all considered $\vartheta $. Define $y_{\vartheta}:=S_{\vartheta}(0,0)$ and $y=S(0,0)$.
Using Lemma \ref{LtauhAbsch} we have
\begin{equation*}
\|y_\vartheta\|_{C(\overline{I};L^2(\Omega))} \leq c\left(\|y_0\|_{H^1_0(\Omega)}+\|y_1\|_{L^2(\Omega)}\right).
\end{equation*}
Thus, the discrete states $y_\vartheta$ are bounded in $L^2(\Omega_T)$. Hence $\lbrace J_\vartheta(0)\rbrace_\vartheta$ is bounded in $\mathbb{R}$. This implies that $J_\vartheta(\overline{u}_{\vartheta})$ is bounded and thus, $( \overline{y}_\vartheta)_\vartheta$ and $( D_t\overline{u}_{\vartheta}=\overline{v}_{\vartheta} )_{\vartheta}$ are bounded in $L^2(\Omega_T)$, and $M(I)^m$ respectively.
Now it suffices to show that $\overline{c}_{\vartheta}\in \mathbb{R}^m$ is bounded in order to get the boundedness of $(\overline{u}_{\vartheta} )_{\vartheta}$ in $BV(I)^m$. Assume that $\overline{c}_{\vartheta}\in \mathbb{R}^m$ is unbounded for $\vartheta \rightarrow 0$. It holds
\begin{equation*}
\sum\nolimits_{j=1}^m\alpha_j \|D_t\hat{u}_{\vartheta,j} \|_{M(I)}=\sum\nolimits_{j=1}^m\alpha_j \|D_t\overline{u}_{\vartheta,j} \|_{M(I)}\leq J_\vartheta(\overline{u}_{\vartheta})\leq J_\vartheta(0)
\end{equation*}
and with the Poincare inequality for $BV(I)$ functions (\cite[p. 152]{[AmFuPa]}), we get that $( \hat{u}_{\vartheta} )_{\vartheta}$ is bounded in $BV(I)^m$. Consider $z_\vartheta=\overline{y}_\vartheta-\widetilde{y}_\vartheta$ with $\widetilde{y}_\vartheta=L_{\vartheta}(\hat{u}_{\vartheta}\cdot g)+Q_{\vartheta}(y_0,y_1)$. The $BV$ boundedness of $(\hat{u}_{\vartheta})_{\vartheta}$, and therefore the boundedness in $L^2(I)^m$, implies by Lemma \ref{LtauhAbsch} that $(\widetilde{y}_\vartheta)_\vartheta$ is bounded in $L^2(\Omega_T)$.
The boundedness of $(\widetilde{y}_\vartheta)_\vartheta$ and $(\overline{y}_\vartheta)_\vartheta$ lead to the boundedness of $(z_\vartheta)_\vartheta$ in $L^2(\Omega_T)$.
The linearity of $L_{\vartheta}(B(\cdot,\cdot))$, implies $z_\vartheta=L_\vartheta(B(0,\overline{c}_{\vartheta}))$. Consider now $p_{\vartheta}:=\max_{1\leq j \leq m}\vert \overline{c}_{\vartheta,j} \vert $, with $\overline{c}_{\vartheta} \to\infty$, $\xi_\vartheta:=\frac{1}{p_{\vartheta}} z_\vartheta$,  and $\overline{a}_{\vartheta}=\frac{\overline{c}_{\vartheta}}{p_{\vartheta}}$. There exists a $\vartheta_0>0$ such that for all $\vartheta<\vartheta_0$ the sequence $\overline{a}_{\vartheta}$ is bounded by definition in $\mathbb{R}^m$. Thus, let us now consider $\vartheta\leq\vartheta_0$ for all sequences in this proof. Hence, there exists a subsequence of $\overline{a}_{\vartheta}$, which converges to some $\overline{a}$. Denote this converging subsequence again by $\overline{a}_{\vartheta}$. The linear structure of $L_{\vartheta}(B(\cdot,\cdot))$ gives us $\xi_\vartheta=L_{\vartheta}(B(0,\overline{a}_{\vartheta}))$.
Define by $\overline{\xi}_\vartheta$ the solution $L_\vartheta(B(0,\overline{a}))$. Next we show that Lemma \ref{LtauhAbsch} leads to
$
\|\xi_\vartheta - \overline{\xi}_\vartheta\|_{L^2(\Omega_T)}\to 0
$.
Thus, we have
\begin{equation*}
\|\xi_\vartheta - \overline{\xi}_\vartheta\|_{L^2(\Omega_T)}=\|L_\vartheta(B(0,\overline{a}_{\vartheta}-\overline{a}))\|_{L^2(\Omega_T)}
\leq c |\overline{a}_{\vartheta}-\overline{a}|_{\mathbb R^m} \to 0.
\end{equation*}
Define $\overline{\xi}$ by $L(B(0,\overline{a}))$. Then we have
\begin{equation*}
\|\overline{\xi}-\overline{\xi}_\vartheta\|_{C(\overline{I};L^2(\Omega))}=\|L(B(0,\overline{a})-L_\vartheta(B(0,\overline{a}))\|_{C(\overline{I};L^2(\Omega))}\to 0
\end{equation*}
according to Theorem \ref{AprioriThm}.
With the boundedness of $z_\vartheta$ in $L^2(\Omega_T)$, the unboundedness of $\vert p_{\vartheta}\vert$, and the definition of $\xi_\vartheta=\frac{z_\vartheta}{p_\vartheta}$, we can deduce that $\xi_\vartheta \rightarrow 0$ in $L^2(\Omega_T)$. Hence, $\overline{\xi}_\vartheta= \xi_\vartheta + (\overline{\xi}_\vartheta -\xi_\vartheta)\rightarrow 0$ in $L^2(\Omega_T)$.
Thus, we obtain that $\overline{\xi}=0$, which implies $\sum\nolimits_{j=1}^m \overline{a}_jg_j=0$. Because $g_j \in L^\infty(\Omega)\setminus\lbrace 0 \rbrace$ have pointwise disjoint supports, we get that $\overline{a}_j=0$, which is a contradiction. Thus, it is shown that $\overline{c}_{\vartheta}$ is bounded, and hence $( \overline{u}_{\vartheta})_{\vartheta}$ is bounded in $BV(I)^m$.
\end{proof}
The next theorem states an a priori error estimate for the optimal state. Under additional assumptions on the structure of optimal adjoint state an improved rate for the optimal state is proven. Furthermore, an optimal convergence for the control in the $L^1(I)$-norm is proven. In order to obtain optimal convergence rates we assume the following regularity on the data:
\begin{assumption}
We assume that
\begin{itemize}
  \item $y_d \in C^1(\overline{I};H^1_0(\Omega))$
  \item $g\in (\mathbb{H}^{2})^m$
  \item $(y_0,y_1)\in \mathbb{H}^{3}\times \mathbb{H}^{2}$
\end{itemize}
\end{assumption}
\begin{corollary}\label{cor:Breg}
There holds: $B(v,c)\in L^\infty(I;\mathbb H^2)$ for any $(v,c)\in M(I)^m\times \mathbb R^m$
\end{corollary}
\begin{proof}
This follows from the definition of $B(v,c)$ and the embedding of $BV(I)$ in $L^\infty(I)$.
\end{proof}
\begin{theorem}\label{convStatesExactVsSemi}
The following non-optimal error rate holds:
\begin{equation}\label{specificRate}
\|\overline{y}-\overline{y}_{\vartheta}\|_{L^2(\Omega_T)}\leq c(h+\tau)\left(\|y_0\|_{\mathbb{H}^{3}} +\|y_1\|_{\mathbb{H}^{2}}+\|y_d\|_{C^1(\overline{I};H^1_0(\Omega)))}\right).
\end{equation}
\end{theorem}
\begin{proof}
By Theorem \ref{energy} we have $\overline{y}-y_d \in C^1(\overline{I};H^1_0(\Omega))$. Using (\ref{AprioriEsti}), implies
\begin{equation}\label{dssdadsda}
\|\overline{p}-\hat{p}_{\vartheta}\|_{L^\infty(I;L^2(\Omega))}^{\frac{1}{2}}\leq
c(h^2 +\tau^2)^{\frac{1}{2}}\| \overline{y}-y_d\|_{W^{1,1}(I;H^1_0(\Omega))}^{\frac12}.
\end{equation}
Corollary \eqref{cor:Breg} and \ref{Cor221} lead to $\|\overline{y}-\hat{y}_{\vartheta}\|_{L^2(\Omega_T)}= O(\tau^2+h^2)$. Hence, Lemma \ref{boundedDiscreteoptiInBVandLInfty} and (\ref{ineq1}) imply (\ref{specificRate}).
\end{proof}
\begin{theorem}\label{CostminusCostDiscRate}
For the optimal control $(\overline{v},\overline{c})$ of $(\tilde P)$ and solutions $(\overline{v}_\vartheta,\overline{c}_\vartheta)$ of $(\tilde{P}^{\text{semi}}_{\vartheta})$ the following a priori error estimates hold:
\begin{equation}\label{vffs}
\vert J(\overline v,\overline c)-J_{\vartheta}(\overline v_\vartheta,\overline c_\vartheta)\vert = O(\tau^2+\tau^2)
\end{equation}
\begin{equation}\label{vffs2}
\sum_{i=1}^{m}\bigg\vert \|\overline v_i\|_{M(I)}-\|\overline v_{\vartheta,i}\|_{M(I)}\bigg\vert = O(\tau+h)
\end{equation}
\end{theorem}
The proof of Theorem \ref{CostminusCostDiscRate} is a modified version of the proof of \cite[Theorem 4.2.]{[VexPie]}.\\
\begin{proof}
Optimality leads to the following two inequalities
\begin{equation*}
J(\overline v,\overline c)\leq J(\overline v_{\vartheta},\overline c_\vartheta)\quad\text{and}\quad J_{\vartheta}(\overline v_{\vartheta},\overline c_\vartheta)\leq J_{\vartheta}(\overline v,\overline c).
\end{equation*}
This implies
$J(\overline v,\overline c)-J_{\vartheta}(\overline v,\overline c)\leq J(\overline v,\overline c)-J_{\vartheta}(\overline v_{\vartheta},\overline c_\vartheta)\leq J(\overline v_{\vartheta},\overline c_\vartheta)-J_{\vartheta}(\overline v_{\vartheta},\overline c_\vartheta)$.
So it remains to estimate the error with respect to the cost functionals for a fixed $(v,c)$, i.e. $(\overline v,\overline c)$ and $(\overline v_\vartheta,\overline c_\vartheta)$. Common calculations lead to the following estimate:
\begin{multline}\label{245f}
|J(v,c)-J_{\vartheta}(v,c)|\\
=\frac{1}{2}\|S(v,c)-S_{\vartheta}(v,c)\|_{L^2(\Omega_T)}^2+|(S(v,c)-S_\vartheta(v,c),S(v,c)-y_d))_{L^2(\Omega_T)}|\\
\leq \frac{1}{2}\|S(v,c)-S_{\vartheta}(v,c)\|_{L^2(\Omega_T)}^2+\|S(v,c)-S_\vartheta(v,c)\|_{L^2(\Omega_T)}\|S(v,c)-y_d\|_{L^2(\Omega_T)}
\end{multline}
Then Corollary \ref{cor:Breg} and \ref{Cor221} implies the first assertion. Finally, Theorem \ref{convStatesExactVsSemi} implies \eqref{vffs2}.
\end{proof}
\subsection{Optimal Convergence Rates for the Optimal Controls of $(\tilde{P}_{\vartheta}^{\text{semi}})$}
Under certain assumptions we show that the BV-representations $\overline{u}_\vartheta$ of the optimal controls of $(\tilde{P}_{\vartheta}^{semi})$, with respect to $\vartheta$, converge with a specific rate in the $L^1-$norm to the solution $\overline{u}$ of $(P)$.
Further, define the following functions:
\begin{align}
z(t)&:=\partial_t p_{1}(t)= \int\nolimits_\Omega L^* \left( S\left(\overline v,\overline c\right)-y_d \right)(t) g~dx\\
z_{\vartheta}(t)&:=\partial_t p_{1,\vartheta}(t)= \int\nolimits_\Omega L_\vartheta^* \left( S_\vartheta\left(\overline v_\vartheta,\overline c_\vartheta\right)-y_d \right)(t) g~dx
\end{align}
with $(\overline v,\overline c)$ as the optimal control of $(\widetilde{P})$ and $(\overline v_\vartheta,\overline c_\vartheta)$ as optimal control of $(\widetilde{P}_\vartheta^{semi})$. Due to Proposition \ref{support_measure_cor} and Remark \ref{RauteRemark}, it holds that $\supp(\overline{v}_i)\subseteq \lbrace t \vert z_i(t)=0\rbrace$ and $\supp(\overline{v}_{\vartheta,i})\subseteq \lbrace t \vert z_{\vartheta,i}(t)=0\rbrace$.
\begin{lemma}\label{LemmaGMatrix}
The matrix $G:=(\langle Lg_i,Lg_j\rangle_{L^2(\Omega_T)})_{i,j=1}^m\in \mathbb{R}^{m\times m}$ is symmetric and positive definite.
\end{lemma}
\begin{proof}
the matrix $G$ is a Gramian-matrix, which is a consequence of the uniqueness of solutions of the wave equation the fact that $\{g_i\}_{i=1}^m$ is a linear independent system.
\end{proof}
\begin{theorem}\label{weak*conv_reg_prob1}
$\overline{u}_\vartheta$ converges weakly* in $BV(0,T)^m$ to the solution $\overline{u}$ for $\vartheta\rightarrow 0$.
\end{theorem}
\begin{proof}
Let $(\vartheta_n)_{n=1}^\infty=(\tau_n,h_n)$ be a null sequence such that $(\tau_n)_{n=1}^\infty\subset \mathbb{R}^+$, $(h_n)_{n=1}^\infty\subset \mathbb{R}^+$.
Lemma \ref{boundedDiscreteoptiInBVandLInfty} implies that $(\overline{u}_{\vartheta_n})_{n=1}^\infty$ is a bounded sequence in $BV(I)^m$ where $(\overline v_{\vartheta_n},\overline c_{\vartheta_n})$ are optimal controls of $(P^{semi}_{\vartheta_n})$.
The weak* compactness of closed and bounded sets in $BV(I)^m$ implies the existence of a subsequence $(\overline{u}_{\vartheta_{n_k}})_{k}$ which converges weakly* to some $\widetilde{u}\in BV(I)^m$. Hence, $(\overline{u}_{\vartheta_{n_k}})_{k}$ converges in $L^2(I)^m$ to $\widetilde{u}$ and $D_t \overline{u}_{\vartheta_{n_k}}=\overline v_{\vartheta_{n_k}}$ converges weakly* in $M(I)^m$ to $D_t\widetilde{u}$. There exists a unique element $(\widetilde{v},\widetilde{c})\in M(I)^m\times \mathbb{R}^m$ such that $\widetilde{u}=\int\nolimits_0^\cdot ~d\widetilde{v}(s)-\frac{1}{T}\int\nolimits_0^T\int\nolimits_0^t ~d\widetilde{v}(s)~dt+\widetilde{c}$. Due to the weak* l.s.c. of $\|\cdot\|_{M(I)}$ in $M(I)$,
we get
\begin{equation}\label{lowSemiContii}
\liminf_{k \rightarrow \infty} \sum\nolimits_{i=1}^m \alpha_i \|\overline v_{\vartheta_{n_k,i}}\|_{M(I)}\geq \sum\nolimits_{i=1}^m \alpha_i \|D_t\widetilde{u}_i\|_{M(I)}.
\end{equation}
Let us show that
\begin{equation}\label{ONeAim134q}
\lim_{k \rightarrow \infty} \|S_{\vartheta_{n_k}}(\overline v_{\vartheta_{n_k}},\overline c_{\vartheta_{n_k}})-y_d\|_{L^2(\Omega_T)}^2= \|S(\widetilde{v},\widetilde{c})-y_d\|_{L^2(\Omega_T)}^2
\end{equation}
holds. Theorem \ref{AprioriThm}, the stability of $L_{\vartheta_{n_k}}$, see Lemma \ref{LtauhAbsch} and the strong convergence of $\overline u_{\vartheta_k}$ in $L^2(I)$ lead to
\begin{multline*}
\|S_{\vartheta_{n_k}}(\overline{v}_{\vartheta_{n_k}},\overline{c}_{\vartheta_{n_k}})-S(\widetilde{v},\widetilde{c})\|_{L^2(\Omega_T)}\leq
\|S_{\vartheta_{n_k}}(\overline{v}_{\vartheta_{n_k}},\overline{c}_{\vartheta_{n_k}})-S_{\vartheta_{n_k}}(\widetilde{v},\widetilde{c})\|_{L^2(\Omega_T)}\\
+\|S_{\vartheta_{n_k}}(\widetilde{v},\widetilde{c})-S(\widetilde{v},\widetilde{c})\|_{L^2(\Omega_T)}
\leq c\|\overline{u}_{\vartheta_k}-\widetilde{u}\|_{L^2(I)^m}\\
+c(h_{n_k}^2+\tau_{{n_k}}^2)
(\|y_0\|_{\mathbb H^3}+\|y_1\|_{\mathbb H^2}+\|B(\widetilde{v},\widetilde{c})\|_{L^\infty(I;\mathbb H^2)})
\end{multline*}
This leads to (\ref{ONeAim134q}). With (\ref{lowSemiContii}), (\ref{ONeAim134q}) and Theorem \ref{CostminusCostDiscRate} we get
\begin{equation*}
J(\overline{v},\overline c)=\liminf_{k \rightarrow \infty} J_{\vartheta_k}(\overline{v}_{\vartheta_k},\overline{c}_{\vartheta_k})\geq J(\widetilde{v},\widetilde c).
\end{equation*}
The uniqueness of the optimal control of $(P)$ leads to the desired result.
\end{proof}
\begin{corollary}\label{corConvWeak*Control}
There holds $\overline{u}_\vartheta\to \overline{u}$ in $L^2(I)$ for $\vartheta\to 0$.
\end{corollary}
Next we prove pointwise convergence of $z_\vartheta$ and $\partial_t z_\vartheta$.
\begin{lemma}\label{GlmConvZ}
For $\vartheta\rightarrow 0$ we have $\|z_\vartheta - z\|_{L^\infty(I)^m} \rightarrow 0$.
\end{lemma}
\begin{proof}
By Theorem \ref{energy} and Definition \ref{discreteSolutdef}, we have that $z_\vartheta\in C(\overline{I})$ and $z\in C^1(\overline{I})$. Hence, $\|z_\vartheta - z\|_{L^\infty(I)^m}$ is well-defined. There holds that
\begin{multline}\label{finalfinal212313}
\|z_\vartheta - z\|_{L^\infty(I)^m}\
\leq c\sup_{t\in \overline{I}}\left\|L^*_{\vartheta}(S_{\vartheta}(\overline v_{\vartheta},\overline c_\vartheta)-y_d)(t,\cdot)-L^*(S(\overline v,\overline c)-y_d)(t,\cdot)\right\|_{L^1(\Omega)}\\
\leq c \left[\|L^*_{\vartheta}(S_{\vartheta}(\overline v_{\vartheta},\overline c_\vartheta )) -L^*(S(\overline v,\overline c))\|_{C(\overline{I};L^2(\Omega))}+\|L^*_{\vartheta}(y_d) -L^*(y_d)\|_{C(\overline{I};L^2(\Omega))}\right].
\end{multline}
Next we show that $\|z_\vartheta - z\|_{L^\infty(I)^m}\to 0$ holds. According to Theorem \ref{AprioriThm} we obtain
\begin{equation}\label{ydConvRate}
\|L^*_{\vartheta}(y_d) -L^*(y_d)\|_{C(\overline{I};L^2(\Omega))}\leq c(h^2+\tau^2)^\frac{1}{3}\|y_d\|_{C^1(\bar I;H^1_0(\Omega))}.
\end{equation}
Furthermore, we have
\begin{multline}\label{SKonvRatetoSvartheta}
\|L^*_{\vartheta}(S_{\vartheta}(\overline v_{\vartheta},\overline c_\vartheta)) -L^*(S(\overline v,\overline c))\|_{C(\overline{I};L^2(\Omega))}\leq\|L^*_{\vartheta}(S_\vartheta(\overline v_\vartheta,\overline c_\vartheta)) -L_{\vartheta}^*(S(\overline v,\overline c))\|_{C(\overline{I};L^2(\Omega))}\\
+\|L^*_{\vartheta}(S(\overline v,\overline c)) -L^*(S(\overline v,\overline c))\|_{C(\overline{I};L^2(\Omega))}
\end{multline}
For the second term on the right hand side of (\ref{SKonvRatetoSvartheta}), we have using Theorem \ref{AprioriThm}:
\begin{equation}\label{kaska2}
\|L^*_{\vartheta}(S(\overline v,\overline c)) -L^*(S(\overline v,\overline c))\|_{C(\overline{I};L^2(\Omega))} \leq	c(h^2+\tau^2)\|S(\overline v,\overline c)\|_{C^1(\bar I;H^1_0(\Omega))}.
\end{equation}
For the first term on the right hand side of (\ref{SKonvRatetoSvartheta}), we use the stability of $L_\vartheta$ and $L^\ast_\vartheta$ from Lemma \ref{LtauhAbsch} to obtain
\begin{multline}\label{Ineq44RHS}
\|L^*_{\vartheta}(S_\vartheta(\overline v_\vartheta,\overline c_\vartheta)) -	L_{\vartheta}^*(S(\overline v,\overline c))\|_{C(\overline{I};L^2(\Omega))}
\leq c\,\|S_\vartheta(\overline v_\vartheta,\overline c_\vartheta)-S(\overline v,\overline c)\|_{C(\overline{I};L^2(\Omega))} \\
\leq c\,\left(\|S_\vartheta(\overline v_\vartheta,\overline c_\vartheta)-S_\vartheta(\overline v,\overline c)\|_{C(\overline{I};L^2(\Omega))}+\|S_\vartheta(\overline v,\overline c)-S(\overline v,\overline c)\|_{C(\overline{I};L^2(\Omega))}\right) \\
\leq c\|\overline{u}_\vartheta-\overline{u}\|_{L^2(I)}
+\|S_\vartheta(\overline v,\overline c)-S(\overline v,\overline c)\|_{C(\overline{I};L^2(\Omega))}.
\end{multline}
The strong convergence of $\overline{u}_\vartheta$ to $\overline{u}$ in $L^2(I)$, see Corollary \ref{corConvWeak*Control}, and Theorem \ref{AprioriThm} imply that $\|L^*_{\vartheta}(S_\vartheta(\overline v_\vartheta,\overline c_\vartheta)) -
L_{\vartheta}^*(S(\overline v,\overline c))\|_{C(\overline{I};L^2(\Omega))}$ converges to 0 for $\vartheta \rightarrow 0$.  Hence, $\|z_\vartheta - z\|_{L^\infty(I)^m}$ converges to 0 for $\vartheta\rightarrow 0$.
\end{proof}
Let us further define $\|u\|_{C_\tau(I;L^2(\Omega))}:=\max_{t_i\in \overline{w}^\tau}\|u(t_i)\|_{L^2(\Omega)}$ for all $u\in C(\overline{I};L^2(\Omega))$ and the mesh operators
$\overline{\partial}_t w_m=\frac{w_m-w_{m-1}}{\tau}$.
\begin{lemma}\label{Ineq46RHS}
There exists a constant $c>0$ such that the following inequality holds for all $f\in L^1(I;L^2(\Omega))$
\begin{equation}
\|\overline{\partial_t}
L^*_\vartheta(f)
\|_{C_\tau(I;L^2(\Omega))}\leq
c\|f\|_{L^1(I;L^2(\Omega))}.
\end{equation}
\end{lemma}
\begin{proof}
This follows directly from \cite[Theorem 2.1]{[Zlot]}.
\end{proof}
\begin{lemma}\label{aprioridiff}
The following a priori error estimate
\[
\|\overline{\partial_t}(L^\ast(f)-L^*_\vartheta(f))\|_{C_\tau(I;L^2(\Omega))}\leq (h^2+\tau^2)^{\frac 13}\|f\|_{W^{1,1}(I;L^2(\Omega))}
\]
holds for all $f\in W^{1,1}(I;L^2(\Omega))$.
\end{lemma}
\begin{proof}
This follows directly from \cite[Theorem 4.2]{[Zlot]}.
\end{proof}
\begin{lemma}
We have
\begin{equation}\label{ConvDiffz}
\|\overline{\partial_t}(z_\vartheta - z)\|_{C_\tau(I)^m}\xrightarrow{} 0\quad \text{for}\quad \vartheta \to 0.
\end{equation}
\end{lemma}
\begin{proof}
Lemma \ref{Ineq46RHS} implies
\begin{multline}\label{Ineq47RHS}
\|\overline{\partial_t}(z_\vartheta - z)\|_{C_\tau(I)^m}
=\sum\nolimits_{\ell=1}^m\sup_{t_i\in \overline{w}^\tau} \left\vert
\overline{\partial_t}\int\nolimits_\Omega \left(
L^*_\vartheta(S_\vartheta(\overline{v}_\vartheta,\overline{c}_\vartheta)-y_d)(t_i)-L^*(S(\overline{v},\overline{c})-y_d)(t_i)
\right)g_\ell ~dx
\right\vert\\
\leq c\|\overline{\partial_t}\left(
L^*_\vartheta(S_\vartheta(\overline{v}_\vartheta,\overline{c}_\vartheta)-y_d)-L^*(S(\overline{v},\overline{c})-y_d)
\right)\|_{C_\tau(I;L^2(\Omega))}\\
\leq
c\|\overline{\partial_t}\left(
L^*_\vartheta(S_\vartheta(\overline{v}_\vartheta,\overline{c}_\vartheta)-y_d)-L^*_\vartheta(S(\overline{v},\overline{c})-y_d)
\right)\|_{C_\tau(I;L^2(\Omega))}\\
+c\|\overline{\partial_t}\left(
L^*_\vartheta(S(\overline{v},\overline{c})-y_d)-L^*(S(\overline{v},\overline{c})-y_d)
\right)\|_{C_\tau(I;L^2(\Omega))}
\\
\leq c\|S_\vartheta(\overline{v}_\vartheta,\overline{c}_\vartheta)-S(\overline{v},\overline{c})\|_{L^1(I;L^2(\Omega))}\\
+c\|\overline{\partial_t}\left(
L^*_\vartheta(S(\overline{v},\overline{c})-y_d)-L^*(S(\overline{v},\overline{c})-y_d)
\right)\|_{C_\tau(I;L^2(\Omega))}
\end{multline}
The first term on the right hand side of the last inequality converges to 0 for $\vartheta \rightarrow 0 $, e.g. see (\ref{Ineq44RHS}). Because $y_d$ and $S(\overline v,\overline c)\in C^1(\overline{I};L^2(\Omega))$ holds Lemma \ref{aprioridiff} implies that the last term in the last inequality converges to 0 for $\vartheta \rightarrow 0 $. This proves the assertion.
\end{proof}
\begin{lemma}\label{ConvDiffLinfty}
We have
\begin{equation}
\|\partial_t(z_\vartheta - z)\|_{L^\infty(I)^m}\to 0 \text{ for }\vartheta \to 0.
\end{equation}
\end{lemma}
\begin{proof}
At first we define a cell-wise discretization of the derivative of $z$ as follows
\begin{equation}\label{fwa}
\overline{\delta_t}z:= \sum\nolimits_{i=1}^M \frac{z(t_i)-z(t_{i-1})}{\tau} \mathbf{1}_{I_i}, \text{ with }I_i:=(t_{i-1},t_i), \text{ }i=1,\cdots,M.
\end{equation}
Then we proceed with
\begin{equation*}
\|\partial_t (z_\vartheta-z)\|_{L^\infty(I)^m}\leq\|\partial_t z_\vartheta - \overline{\delta_t}z\|_{L^\infty(I)^m}+\|\overline{\delta_t}z - \partial_t z\|_{L^\infty(I)^m}.\quad
\end{equation*}
Using the disjoint supports of the characteristic functions in the definition of $\overline{\delta_t}z$ leads to
\begin{equation*}
\|\partial_t z_\vartheta - \overline{\delta_t}z\|_{L^\infty(I)^m}=\sum_{j=1}^m\max_{i=1,\cdots,M} \left\vert \frac{z_\vartheta^j(t_i)-z_\vartheta^j(t_{i-1})}{\tau} -\frac{z_j(t_i)-z_j(t_{i-1})}{\tau}\right\vert=\|\overline{\partial}_t(z_\vartheta -z)\|_{C_\tau(I)^m}
\end{equation*}	
which converges to $0$ under the consideration of (\ref{ConvDiffz}).
Further, calculations show that
\begin{multline*}
\|\overline{\delta_t}z - \partial_t z\|_{L^\infty(I)^m}=\left\|\sum\nolimits_{i=1}^M \left(\frac{z(t_i)-z(t_{i-1})}{\tau} - \partial_tz(t)\right)\mathbf{1}_{I_i}(t)\right\|_{L^\infty(I)^m}\\
\leq\left\|\sum\nolimits_{i=1}^M \left(\frac{z(t_i)-z(t_{i-1})}{\tau} - \partial_t z(t_i)\right)\mathbf{1}_{I_i}(t)\right\|_{L^\infty(I)^m}+
\left\|\sum\nolimits_{i=1}^M \left(\partial_t z(t_i) - \partial_t z(t)\right)\mathbf{1}_{I_i}(t)\right\|_{L^\infty(I)^m}\\
=\sum_{j=1}^m\left(\max_{i=1,\cdots,M}\left\vert\frac{z_j(t_i)-z_j(t_{i-1})}{\tau} - \partial_t z_j(t_i)\right\vert+
\max_{i=1,\cdots,M}\sup_{t\in I_i}\vert \partial_t z_j(t_i) - \partial_t z_j(t)\vert\right).
\end{multline*}
In the last equation, we directly see that the first term converges to $0$ due to \cite[Theorem 1.11]{anastassiou2017intelligent}. The second term converges to $0$ due to the uniform continuity of $\partial_tz(t)$ in $\overline{I}$.
Hence, the result follows for $\vartheta \rightarrow 0$, which implies the claim.
\end{proof}
\begin{lemma}\label{Ineq49RHS}
	The convergence $\|p_{1,\vartheta} - p_1\|_{L^\infty(I)^m} \rightarrow 0$ holds for $\vartheta\rightarrow 0$.
\end{lemma}
\begin{proof}
This follows directly from
\[
\|p_{1,\vartheta} - p_1\|_{L^\infty(I)^m} \leq c\|L_\vartheta^*(S_\vartheta(\overline{v}_\vartheta,\overline c_\vartheta)-y_d)-L^*(S(\overline{v},\overline c)-y_d)\|_{C(\overline{I};L^2(\Omega))},
\]
 which converges to $0$ using the same steps as in Lemma \ref{GlmConvZ}.
\end{proof}
In order to proof a priori error estimates for the control in the $L^1(I)$-norm and higher convergence rates for the state variable we have to make the following assumption.
\begin{assumption}\label{as:jumps}
\begin{itemize}
\item[(A1)] $\lbrace t\in I\vert \vert \overline p_{1,i}(t)\vert=\alpha\rbrace = \lbrace t_{1,i}, \cdots , t_{m_i,i}\rbrace$ for $m_i\in \mathbb{N}$, with $i=1,\cdots, m$.\label{assumptionFiniteJumps1}
\item[(A2)] $\partial_t z_i(t_{j,i})\neq 0$, for $j=1,\cdots,m_i$ and $i=1,\cdots,m$.\label{assumptionFiniteJumps2}
\end{itemize}
\end{assumption}
\begin{remark}
The assumption (A1) enforces finitely many jumps for the optimal control of $(P)$, i.e. it holds $\supp(D_t \overline{u}_i)\subseteq \lbrace t_{1,i}, \cdots , t_{m_i,i}\rbrace$ for $\overline{u}\in BV(I)^m$.
\end{remark}
\begin{lemma}\label{coalaLemma} Let $(\overline{v}_\vartheta,\overline{c}_\vartheta)$ be an optimal control of $(\widetilde{P}^{semi}_\vartheta)$.
Under the assumptions (A1) and (A2) above, there exists a $\delta >0$, and $\vartheta_0 >0$ such that for all $0<\vartheta \leq \vartheta_0$ holds
\begin{equation*}
\overline{v}_{\vartheta,i}=\sum\nolimits_{l=1}^{m_i} c_{l,\vartheta}^i\delta_{t_{l,\vartheta}^i} \text{ with }c_{l,\vartheta}^i\in \mathbb{R} \text{ and }
t_{l,\vartheta}^i\in B_\delta(t_{l,i}),
\end{equation*}
where $B_\delta(t_{l,i})$ are pairwise disjoint for a fixed $i=1,\cdots,m$ with respect to the index $l=1,\cdots,m_i$ and with $0<\vartheta\rightarrow 0$. The coefficients in front of the Dirac measures of $\overline{v}_{\vartheta,i}$, i.e. $c_{l,\vartheta}^i$ for $l=1,\cdots,m_i$, are possibly $0$.
\end{lemma}
\begin{proof}
Let us begin with the case $m=1$, $m_1=1$, i.e. $\lbrace t\in I \vert\vert p_1(t)\vert=\alpha\rbrace=\lbrace \tilde{t}_1\rbrace$. First of all we know that $|p_1(t)|\leq \alpha$ for all $t\in \overline I$ holds and since $\overline p_1\in C^1(\overline I)$ as well as that $\tilde t_1$ is an interior point follows $z(\tilde t_1)=-\partial_tp_1(\tilde t_1)=0$. Moreover, due to (A2) there exists a $\delta >0$ and $c_1>0$ such that $|\partial_t z(t)|>c_1$ for all $t\in B_\delta(\tilde t_1)\subset I$. Since $\partial_t z$ is continuous, $\partial_t z$ does not change its sign on $B_\delta(\tilde t)$ and hence $z$ is strictly monotone in $B_\delta(\tilde t_1)$. Therefore $\tilde t_1$ is the only root of $z$ in $B_\delta(\tilde t_1)$. Moreover, there exist $t_-,t_+\in B_\delta(\tilde t_1)$ with $z(t_-)<0<z(t_+)$. By Lemma \ref{GlmConvZ} there exists a $\vartheta_0=(\tau_0,h_0)$ such that $z_\vartheta(t_-)<0<z_\vartheta(t_+)$ for all $\vartheta<\vartheta_0$. Since $z_\vartheta$ is continuous there exists a $t_\vartheta\in (t_-,t_+)$ such that $z_\vartheta(t_\vartheta)=0$ for all $\vartheta<\vartheta_0$. Next we show that there exists a $\widetilde{\vartheta}_0\leq \vartheta_0$ such that $t_{\vartheta}$ is the only root of $z_\vartheta$ in $B_\delta(\tilde t_1)$ for all $\vartheta<\tilde \vartheta_0$. Lemma \ref{ConvDiffLinfty} implies existence of a $\tilde \vartheta_0<\vartheta_0$ that $\partial_t z_\vartheta$ is either strictly positive or strictly negative on $B_\delta(t_1)$. Now let $\hat t_\vartheta$ be a second root of $z_\vartheta$ in $B_\delta(\tilde t_1)$. Then it holds
\[
0=z_\vartheta(t_\vartheta)-z_\vartheta(\hat t_\vartheta)=\int\nolimits_{\hat t_\vartheta}^{t_\vartheta}\partial_tz_\vartheta(t)~ dt\neq 0.
\]
Hence, there is no second root of $z_\vartheta$ in $B_\delta(\tilde t_1)$. Next we show that $t\neq t_\vartheta$ and $z_\vartheta(t)=0$ imply the existence of $\hat \vartheta_0<\tilde \vartheta_0$ such that $|p_{1,\vartheta}(t)|<\alpha$ for all $\vartheta<\hat\vartheta$ and thus $\overline v_\vartheta=c_{1,\vartheta}\delta_{t_\vartheta}$ with $c_1$ possibly zero. Such a $t$ can only exist in $\overline I\setminus B_\delta(\tilde t_1)$. Due to Assumption \ref{assumptionFiniteJumps1} and the condition $|p_{1,\vartheta}(t)|\leq\alpha$ for all $t\in \overline I$ there exists a $\varepsilon>0$ such that $|p_{1}(t)|<\alpha-\varepsilon$ for all $t\in \overline I\setminus B_\delta(\tilde t_1)$. Lemma \ref{Ineq49RHS} implies the existence of a $\hat \vartheta_0<\tilde \vartheta$ with $|p_{1,\vartheta}(t)|<\alpha-\varepsilon/2$ for all $\vartheta < \hat \vartheta_0$ and $t\in \overline I\setminus B_\delta(\tilde t_1)$.
In the case of $m=1$ and $\lbrace t\in I \vert\vert p_1(t)\vert=\alpha\rbrace=\lbrace \tilde{t}_1,\cdots , \tilde{t}_{m_1}\rbrace$ with $m_1>1$, we can find for each $\tilde{t}_i$ a $\delta_i>0$ with $\bigcap_{i=1}^{m_1}B_{\delta_i}(\tilde t_i)=\emptyset$ and a $\vartheta_i$ such that there exists a $t_\vartheta^i\in B_{\delta_i}(\tilde t_i)$ with $\overline v_\vartheta|_{B_{\delta_i}(\tilde t_i)}=c_{i,\vartheta}\delta_{t_\vartheta^i}$ and $\overline v_\vartheta|_{\overline I\setminus \cup_{i=1}^{m_1}B_{\delta_i}(\tilde t_i)}=0$. Then choose $\vartheta_0<\min_{i=1,\ldots,m_1}\vartheta_i$.
In the case of $m>1$, one has to consider the same proof as above with respect to an additional subindex $i=1,\cdots,m$, and the smallest $\vartheta_0$ and $\delta$ used in the proofs of each component $i=1,\cdots,m$.
\end{proof}
From now on, we will assume that $\vartheta\leq \vartheta_0$ holds with $\vartheta_0$ from Lemma \ref{coalaLemma}.
Furthermore, without loss of generality, we assume that $\delta>0$ in Lemma \ref{coalaLemma} is considered to be small enough such that there exists a $\widetilde{\delta}>0$ for which $\widetilde{\delta}< dist(B_\delta(t_{j_1,i}),\lbrace 0,T \rbrace)$, and $\widetilde{\delta}< dist(B_\delta(t_{j_1,i}),B_\delta(t_{j_2,i}))$, $j_1,j_2=1,\cdots,m_i$, $j_1\neq j_2$ for $i=1,\cdots,m$. Let us note that Remark \ref{fsa} and Lemma \ref{coalaLemma} guarantee that such a $\widetilde\delta>0$ exists.
Under these assumptions, we can work with the following definition.
\begin{definition}\label{Ineq51RHS}
Let us define the BV representations of the optimal controls of $(P)$ and $(P^{\text{semi}}_\vartheta)$ in a more explicit form
\[
\overline u_i=c_i+\sum\nolimits_{j=1}^{m_i} c_j^{i}\left(\mathbf{1}_{(t_{j,i},T]}(t)-\frac{T-t_{j,i}}{T}\right),\quad \overline u_{i,\vartheta}=c_{i,\vartheta}+
\sum\nolimits_{j=1}^{m_i} c_{j,\vartheta}^{i}\left(\mathbf{1}_{(t_{j,\vartheta}^i,T]}(t)-\frac{T-t_{j,\vartheta}^i}{T}\right)
\]
for $i=1,\ldots,m$.
\end{definition}
\begin{lemma}\label{AssumptionDecompoL1Norm}
The following inequality holds
\begin{equation}\label{BootstrapIneq2}
\|\overline{u}-\overline{u}_\vartheta\|_{L^1(I)^m}\leq c\left(\sum\nolimits_{i=1}^m \left(|c_i-c_{i,\vartheta}|+ \sum\nolimits_{j=1}^{m_i} |c_j^i|\cdot|t_{j,i}-t_{j,\vartheta}^i|+|c_j^i-c_{j,\vartheta}^i|\right)\right)
\end{equation}
for some constant $c$ which depends only on $T$.
\end{lemma}
We can prove (\ref{BootstrapIneq2}) by using $\|\mathbf{1}_{(t_{j,i},T]}-\mathbf{1}_{(t_{j,\vartheta}^i,T]}\|_{L^1(I)}=\vert t_{j,i} - t_{j,\vartheta}^i \vert$.
\begin{lemma}\label{Ineq52RHS}
For each $t_{j,i}$ there is a function $f_{j,i}\in C^\infty_c(\Omega_T)$, with $j=1,\cdots,m_i$ and $i=1,\cdots,m$, such that the function
\begin{equation}\label{PlatooFct}
g_i^j(t,x):=\int\nolimits_{t}^T  L^*(f_{j,i})(s,x)~ds\in C_0(I;L^2(\Omega))
\end{equation}
fulfills the properties
\begin{itemize}
\item[a)] $L^*(f_{j,i})(t,x)=h_{i,j}(t) f_i(x)$ for some $h_{i,j}\in C^\infty_c(I)$, $f_i \in C^\infty_c(\Omega)$, and $f_{j,i}=f_i\partial_{tt}h_{i,j} - h_{i,j}\Delta f_i$,
\item[b)] $0\leq	\int\nolimits_t^T h_{i,j}(s)~ds<1$ in $[t_{j,i}-\delta-\frac{\widetilde{\delta}}{2},t_{j,i}+\delta+\frac{\widetilde{\delta}}{2}] \setminus B_{\delta}(t_{j,i})$,
\item[c)] $\int\nolimits_t^T h_{i,j}(s)~ds=1$ in $B_{\delta}(t_{j,i})$,
\item[d)]
$\supp\left(\int\nolimits_t^T h_{i,j}(s)~ds\right)\subseteq[t_{j,i}-\delta-\frac{\widetilde{\delta}}{2},t_{j,i}+\delta+\frac{\widetilde{\delta}}{2}]\subset\subset I$, i.e. $\int\nolimits_\Omega g_i^j~dx\in C_0(\overline{I})$,
\item[e)]  $\langle f_i,g_l\rangle_{L^2(\Omega)}=\delta_{i,l}$,
\end{itemize}
with $\delta_{i,l}=0$, if $l\neq 0$ and $1$ else.
\end{lemma}
\begin{proof}
For all $I_{j,i}:=[t_{j,i}-\delta-\frac{\widetilde{\delta}}{2},t_{j,i}+\delta+\frac{\widetilde{\delta}}{2}]$, $j=1,\cdots,m_i$ and $i=1,\cdots,m$,  there exists $\widetilde{p}_{j,i}\in C^\infty_c(I)$ such that $0\leq\widetilde{p}_{j,i}\leq 1$ with
\begin{equation}\label{Could_be_generalized}
\widetilde{p}_{j,i}=0 \text{ in } I\setminus I_{j,i}, \quad \widetilde{p}_{j,i}=1 \text{ in } [t_{j,i}-\delta,t_{j,i}+\delta].
\end{equation}
Let us define $h_{i,j}=-\partial_t\widetilde{p}_{j,i}$. For each $g_i$, $i=1,\cdots,m$, we can find a $f_i\in C^\infty_c(\Omega)$ such that $(f_i , g_k )_{L^2(\Omega)}$ is $0$ for $i\neq k$ and $1$ else. One can show that $L^*(f_{j,i})=h_{i,j} f_i$ with $f_{j,i}:=f_i\partial_{tt}h_{i,j} - h_{i,j}\Delta f_i$ holds.
Hence, $g_i^j$, defined in (\ref{PlatooFct}), fulfills the desired properties a)-e).
\end{proof}
\begin{lemma}\label{LemmajumpHeight}
There exists a constant $c>0$ independent of $\vartheta$ such that
\begin{equation}\label{opti1}
\vert c_j^i-c_{j,\vartheta}^i \vert\leq c\left(\tau^2+h^2+\|\overline y-\overline y_{\vartheta}\|_{L^2(\Omega_T)}\right) \text{ with } j=1,\cdots,m_i \text{ and }i=1,\cdots,m.
\end{equation}
\end{lemma}
\begin{proof}
Let $i=1,\cdots,m$, $j\in \lbrace 1,\cdots,m_i \rbrace$ and consider from Lemma \ref{Ineq52RHS} the function $g^j_i=\int\nolimits^T_t h_{i,j}f_i\,ds=\int\nolimits_t^T L^\ast(f_{j,i})\,ds$. Hence, we have
\begin{multline}\label{Stepheight}
c_j^i-c_{j,\vartheta}^i=\int\nolimits_0^T \int\nolimits_t^T  h_{i,j}(s)~ds~dD_t(\overline{u}_i-\overline{u}_{\vartheta,i})(t)=\int\nolimits_0^T h_{i,j} (\overline{u}_i-\overline{u}_{\vartheta,i})~dt\\
=\sum\nolimits_{l=1}^m \int\nolimits_0^T h_{i,j} (\overline{u}_l-\overline{u}_{\vartheta,l})\,dt \int\nolimits_\Omega f_i g_l\,dx=\int\nolimits_0^T\int\nolimits_\Omega
-\partial_t g^j_i \left(\sum\nolimits_{l=1}^m (\overline{u}_l-\overline{u}_{\vartheta,l})g_l\right)~dt~dx
\end{multline}
Thus, it follows
\begin{multline*}
c_j^i-c_{j,\vartheta}^i= \int\nolimits_0^T\int\nolimits_\Omega-\partial_t g^j_i \left(\sum\nolimits_{l=1}^m (\overline{u}_l-\overline{u}_{\vartheta,l}) g_l\right)\,dt\,dx= \int\nolimits_0^T \int\nolimits_{\Omega} L^*(f_{j,i})(B(\overline v,\overline c)-B(\overline v_\vartheta,\overline c_\vartheta))~dt~dx\\
=(L^*(f_{j,i})-L_\vartheta^*(f_{j,i}),B(\overline v,\overline c)-B(\overline v_\vartheta,\overline c_\vartheta))_{L^2(\Omega_T)}+(L_\vartheta^*(f_{j,i}),B(\overline v,\overline c)-B(\overline v_\vartheta,\overline c_\vartheta))_{L^2(\Omega_T)}.
\end{multline*}
By Theorem \ref{AprioriThm} and the boundedness of $(\overline v_\vartheta,\overline c_\vartheta)$ we obtain
\begin{multline*}
(L^*(f_{j,i})-L_\vartheta^*(f_{j,i}),B(\overline v,\overline c)-B(\overline v_\vartheta,\overline c_\vartheta))_{L^2(\Omega_T)}\\
\leq \|L^*(f_{j,i})-L_\vartheta^*(f_{j,i})\|_{L^2(\Omega_T)}\|B(\overline v,\overline c)-B(\overline v_\vartheta,\overline c_\vartheta)\|_{L^2(\Omega_T)}=O(\tau^2+h^2).
\end{multline*}
Moreover, there holds
\begin{multline*}
(L_\vartheta^*(f_{j,i}),B(\overline v,\overline c)-B(\overline v_\vartheta,\overline c_\vartheta))_{L^2(\Omega_T)}
=\left(f_{j,i},L_\vartheta(B(\overline v,\overline c)) - L_\vartheta(B(\overline v_\vartheta,\overline c_\vartheta))\right)_{L^2(\Omega_T)}\\
\leq \|f_{j,i}\|_{L^2(\Omega_T)}\|S_\vartheta(\overline v,\overline c) -S_\vartheta(\overline v_\vartheta,\overline c_\vartheta)\|_{L^2(\Omega_T)}\\
\leq c\|S_\vartheta(\overline v,\overline c)-S(\overline v,\overline c)\|_{L^2(\Omega_T)} + c\|S(\overline v,\overline c)-S_\vartheta(\overline v_\vartheta,\overline c_\vartheta)\|_{L^2(\Omega_T)}\\
\leq c(\tau^2+h^2+\|\overline{y}-\overline{y}_\vartheta \|_{L^2(\Omega_T)})
\end{multline*}
according to Corollary \ref{cor:Breg} and \ref{Cor221}.
\end{proof}
\begin{lemma}\label{First111Lemma}
There holds that
\begin{equation}\label{First111}
\vert t_{j,1}^i-t_{j,\vartheta}^i \vert\leq c\left(\tau^2+h^2+\|\overline y-\overline y_{\vartheta}\|_{L^2(\Omega_T)}\right) \text{ with }j=1,\cdots,m_i \text{ and } i=1,\cdots,m .
\end{equation}
\end{lemma}
\begin{proof}
Using that $z_i(t_{j,i})=z_{\vartheta,i}(t_{j,\vartheta}^i)=0$ and $z_i=\partial_t p_{1,i}\in C^1(\overline{I})$ gives us
\(
z_i(t_{j,\vartheta}^i)-z_{\vartheta,i}(t_{j,\vartheta}^i)=z_i(t_{j,\vartheta}^i)=z_i(t_{j,i})+\partial_t z_i(\xi)(t_{j,i}-t_{j,\vartheta}^i)
\)
for some $\xi\in B_\delta(t_{j,i})$. In the proof of Lemma \ref{coalaLemma} we have shown that $\vert\partial_t z_i(\tilde{\xi})\vert>0$ for all $\tilde{\xi} \in B_\delta(t_{j,i})$ and therefore we have $\partial_t z_i(\xi)\neq 0$.
Then Lemma \ref{bddlinOpLtauh} and Theorem \ref{AprioriThm} imply
\begin{multline}\label{sdskey}
\vert t_{j,i}-t_{j,\vartheta}^i \vert \leq c \|z_i-z_{\vartheta,i}\|_{L^\infty(I)}\leq c \|z-z_{\vartheta}\|_{L^\infty(I)^m}\\
\leq c \|L^*(S(\overline v,\overline c)-y_d)-L_\vartheta^*(S_\vartheta(\overline v_\vartheta,\overline c_\vartheta)-y_d)\|_{C(\overline{I};L^2(\Omega))}\\
\leq c\|L^*(S(\overline v,\overline c)-y_d)-L_\vartheta^*(S(\overline v,\overline c)-y_d)\|_{C(\overline{I};L^2(\Omega))}+c\|L_\vartheta^*(S(\overline v,\overline c)-S_\vartheta(\overline v_\vartheta,\overline c_\vartheta))\|_{C(\overline{I};L^2(\Omega))}\\
\leq c\|L^*(S(\overline v,\overline c)-y_d)-L_\vartheta^*(S(\overline v,\overline c)-y_d)\|_{C(\overline{I};L^2(\Omega))}+c\|S(\overline v,\overline c)-S_\vartheta(\overline v_\vartheta,\overline c_\vartheta)\|_{L^1(I;L^2(\Omega))}\\
\leq c(\tau^2+h^2+\|S(\overline v,\overline c)-S_\vartheta(\overline v_\vartheta,\overline c_\vartheta)\|_{L^2(\Omega_T)}).
\end{multline}
\end{proof}
\begin{lemma}\label{constantsConvRate}
Let $\vartheta_0$ be small enough such that
\[
\left(\left(L_{\vartheta_0}(g_i),L_{\vartheta_0}(g_j)\right)_{L^2(\Omega_T)}\right)_{i,j=\lbrace 1,\cdots, m \rbrace}>0 \text{ holds}.
\]
Then we obtain
\begin{equation}\label{First112}
\vert c_{i}-c_{i,\vartheta} \vert \leq c(\tau^2 +h^2 + \|\overline{y}-\overline{y}_{\vartheta}\|_{L^2(\Omega_T)}).
\end{equation}
\end{lemma}
\begin{proof}
First, define the function $h(t)=\mathbf{1}_{(t,T]}-\frac{T-t}{T}$. The optimality conditions of the continuous and discrete problem lead to
\begin{multline*}
0=p_{1,i}(0)=\sum\nolimits_{l=1}^m c_l\int\nolimits_0^T\int\nolimits_\Omega L^* \left( L(g_l)\right)g_i~dx~dt\\
+\sum\nolimits_{l=1}^m\sum\nolimits_{j=1}^{m_l} c_j^{l}\int\nolimits_0^T\int\nolimits_\Omega L^* \left( L\left(h(t_{j,l})g_l\right)\right)g_i~dx~dt
+\int\nolimits_0^T\int\nolimits_\Omega L^* \left(Q(y_0,y_1)-y_d \right)g_i~dx~dt
\end{multline*}
as well as in the discrete case to
\begin{multline*}
0=p_{1,\vartheta,i}(0)=\sum\nolimits_{l=1}^m c_{l,\vartheta}\int\nolimits_0^T\int\nolimits_\Omega L_\vartheta^* \left( L_\vartheta(g_l)\right)g_i~dx~dt\\
+\sum\nolimits_{l=1}^m\sum\nolimits_{j=1}^{m_l} c_{j,\vartheta}^{l}\int\nolimits_0^T\int\nolimits_\Omega L_\vartheta^* \left( L_\vartheta\left(h(t_{j,\vartheta}^i)g_l\right)\right) g_i~dx~dt
+\int\nolimits_0^T\int\nolimits_\Omega L_\vartheta^* \left(Q_\vartheta(y_0,y_1)-y_d \right) g_i~dx~dt.
\end{multline*}
By taking the difference of the last two terms we get
\begin{multline}\label{dNumber}
\sum\nolimits_{l=1}^m(c_{l,\vartheta}-c_l)\int\nolimits_0^T\int\nolimits_\Omega L^* \left( L(g_l)\right)g_i~dx~dt=\sum\nolimits_{l=1}^m c_{l,\vartheta}\int\nolimits_0^T\int\nolimits_\Omega \left[L^* \left( L(g_l)\right) -  L_\vartheta^* \left( L_\vartheta(g_l)\right)\right]g_i~dx~dt\\
+\sum\nolimits_{l=1}^m\sum\nolimits_{j=1}^{m_l} (c_j^{l}-c_{j,\vartheta}^l)\int\nolimits_0^T\int\nolimits_\Omega L^* \left(L\left(h(t_{j,l})g_l\right)\right)g_i~dx~dt\\
+\sum\nolimits_{l=1}^m\sum\nolimits_{j=1}^{m_l} c_{j,\vartheta}^l\left(\int\nolimits_0^T\int\nolimits_\Omega \left[ L^* \left( L\left(h(t_{j,l})g_l\right)\right)-L_\vartheta^* \left( L_\vartheta\left(h(t_{j,\vartheta}^i)g_l\right)\right)\right]g_i~dx~dt\right)\\
+\int\nolimits_0^T\int\nolimits_\Omega \left[L^* \left(Q(y_0,y_1)-y_d \right)-L_\vartheta^* \left(Q_\vartheta(y_0,y_1)-y_d \right)\right]g_i~dx~dt.
\end{multline}
For the following we remark that $\bar u_\vartheta$ is bounded $BV(I)$, see Lemma \ref{boundedDiscreteoptiInBVandLInfty}. Then we consider the first term in \eqref{dNumber} on the righthand side. The regularity of $g_l$ implies that $L(g_l)\in C^1(\overline I;H^1_0(\Omega))$ according to Theorem \ref{energy}. Thus, with (\ref{AprioriEsti}), Lemma \ref{bddlinOpLtauh}, Corollary \ref{cor:Breg} and \ref{Cor221} it follows,
\begin{multline*}
\sum\nolimits_{l=1}^m c_{l,\vartheta}\int\nolimits_0^T\int\nolimits_\Omega \left[L^* \left( L(g_l)\right) -  L_\vartheta^* \left( L_\vartheta(g_l)\right)\right]g_i~dx~dt\leq\sum\nolimits_{l=1}^m c\| L^*(L(g_l))-L_\vartheta^*(L_\vartheta(g_l))\|_{L^2(\Omega_T)}\\
\leq c\sum\nolimits_{l=1}^m\left(\| L^*(L(g_l))-L_\vartheta^*(L(g_l))\|_{L^2(\Omega_T)}+\| L(g_l)-L_\vartheta(g_l)\|_{L^2(\Omega_T)}\right)=O(\tau^2 + h^2).
\end{multline*}
By Lemma \ref{LemmajumpHeight} we obtain:
\begin{multline}\label{bootstrapNeededLabel}
\sum\nolimits_{l=1}^m\sum\nolimits_{j=1}^{m_l} (c_j^{l}-c_{j,\vartheta}^l)\int\nolimits_0^T\int\nolimits_\Omega L^* \left(L\left(h(t_{j,l})g_l\right)\right)g_i~dx~dt\\
\leq c\sum\nolimits_{l=1}^m\sum\nolimits_{j=1}^{m_l} |c_j^{l}-c_{j,\vartheta}^l|\leq c\left(\tau^2+h^2+\|\overline y-\overline y_{\vartheta}\|_{L^2(\Omega_T)}\right)
\end{multline}
Now we consider the third term on the righthand side of \eqref{dNumber}. The stability of $L^\ast_\vartheta$, see Lemma \ref{bddlinOpLtauh}, imply
\begin{multline}\label{(75)}
\sum\nolimits_{l=1}^m\sum\nolimits_{j=1}^{m_l} c_{j,\vartheta}^l\left(\int\nolimits_0^T\int\nolimits_\Omega \left[ L^* \left( L\left(h(t_{j,l})g_l\right)\right)-L_\vartheta^*\left(L_\vartheta\left(h(t_{j,\vartheta}^l)g_l\right)\right)\right]g_i~dx~dt\right)\\
\leq c\sum\nolimits_{l=1}^m\sum\nolimits_{j=1}^{m_l} \left\|L^*\left(L\left(h(t_{j,l})g_l\right)\right)-L_\vartheta^* \left(L_\vartheta\left(h(t_{j,\vartheta}^l)g_l\right)\right)\right\|_{L^1(\Omega_T)}
\\
\leq c\sum\nolimits_{l=1}^m\sum\nolimits_{j=1}^{m_l} \left\|L^* \left( L\left(h(t_{j,l})g_l\right)\right)-L_\vartheta^* \left( L\left(h(t_{j,l}^j)g_l\right)\right)\right\|_{L^1(\Omega_T)} +
\left\|L\left(h(t_{j,l})g_l\right)-L_\vartheta\left(h(t_{j,\vartheta}^l)g_l\right)\right\|_{L^1(\Omega_T)}
\end{multline}
Again, by Theorem \ref{energy} we have $L(h(\tilde t)g_l)\in C^1(\overline I;H^1_0(\Omega))$ and any $\tilde t\in I$. Hence, with (\ref{AprioriEsti}) we get
\begin{equation*}
\sum\nolimits_{l=1}^m\sum\nolimits_{j=1}^{m_l} \|L^* \left(L\left(h(t_{j,l})g_l\right)\right)-L_\vartheta^*\left(L\left(h(t_{j,l})g_l\right)\right)\|_{L^1(\Omega_T)} =O(\tau^2 + h^2).
\end{equation*}
Next we consider the following inequality
\begin{multline}\label{dsdsfsfs}
\|L\left(h(t_{j,l}g_l)\right)-L_\vartheta\left(h(t_{j,\vartheta}^l)g_l\right)\|_{L^1(\Omega_T)}\\
\leq\|L\left(h(t_{j,l})g_l\right)-L_\vartheta\left(h(t_{j,\vartheta}^l)g_l\right)\|_{L^1(\Omega_T)}+\|L_\vartheta\left(h(t_{j,l})g_l\right)-L_\vartheta\left(h(t_{j,\vartheta}^l)g_l\right)\|_{L^1(\Omega_T)}.
\end{multline}
Due to Corollary \ref{cor:Breg} and \ref{Cor221}, the first term on the right hand side of (\ref{dsdsfsfs}) possess the asymptotic rate $O(\tau^2+h^2)$. By Lemma \ref{LtauhAbsch}, and Lemma \ref{First111Lemma} we obtain for the second term an estimate in terms of $c(\tau^2 +h^2 + \|\overline{y}-\overline{y}_{\vartheta}\|_{L^2(\Omega_T)})$.
Finally, we consider the last term in \eqref{dNumber}. We have
\begin{multline}\label{ineqality}
\int\nolimits_0^T\int\nolimits_\Omega \left[L^* \left(Q(y_0,y_1)-y_d \right)-L_\vartheta^* \left(Q_\vartheta(y_0,y_1)-y_d \right)\right]g_i~dx~dt\\
\leq c\|L^* \left(Q(y_0,y_1)-y_d \right)-L_\vartheta^* \left(Q(y_0,y_1)-y_d \right)\|_{L^2(\Omega_T)}\\
+c\|L_\vartheta^* \left(Q(y_0,y_1)-y_d \right)-L_\vartheta^* \left(Q_\vartheta(y_0,y_1)-y_d \right)\|_{L^2(\Omega_T)}.
\end{multline}
The first term converges in \eqref{ineqality} with a rate $(\tau^2+h^2)$ according to Theorem \ref{AprioriThm} since $Q(y_0,y_1)-y_d\in C^1(\overline I,H^1_0(\Omega))$. The prescribed regularity of $(y_0,y_1)$, Lemma \ref{LtauhAbsch} and the error estimates in (\ref{AprioriEsti}) give us an estimate in terms of order $(\tau^2 + h^2)$ of the last term in \eqref{ineqality}. Thus, we have
\begin{equation}\label{RateOrderG}
 \sum\nolimits_{l=1}^m(c_l-c_{l,\vartheta})\int\nolimits_0^T\int\nolimits_\Omega L^* \left( L(g_l)\right)g_i~dx~dt \leq c(\tau^2 + h^2+\|\overline{y}-\overline{y}_{\vartheta}\|_{L^2(\Omega_T)}).
\end{equation}
Next we recall the symmetric positive definiteness of the matrix $G$ from Lemma \ref{LemmaGMatrix}. It holds
\begin{equation*}
G(\overline{c}-\overline{c}_\vartheta)=\left(\sum\nolimits_{l=1}^m(c_l-c_{l,\vartheta})\int\nolimits_0^T\int\nolimits_\Omega L^* \left( L(g_l)\right)g_i~dx~dt\right)_{i=1}^m.
\end{equation*}
Furthermore, we have
\begin{equation}\label{GlueG}
\|G(\overline{c}-\overline{c}_\vartheta)\|_{\mathbb{R}^m}\geq \lambda_{min} \|\overline{c}-\overline{c}_\vartheta\|_{\mathbb{R}^m} \geq c\lambda_{min}\|\overline{c}-\overline{c}_\vartheta\|_\infty \geq c\lambda_{min}|c_i-c_{i,\vartheta}|
\end{equation}
for $i=1,\cdots,m$ where $\lambda_{min}>0$ is the smallest eigenvalue of $G$.
Using (\ref{GlueG}) and the convergence rates in (\ref{RateOrderG}) gives us (\ref{First112}).
\end{proof}
From now on we assume that all assumptions in Lemma \ref{constantsConvRate} hold.
\begin{corollary}\label{corRate}
It holds that
\[
\|\overline{u}-\overline u_\vartheta\|_{L^1(I)}\leq c(\tau^2+h^2+\|\overline y-\overline y_\vartheta\|_{L^2(\Omega_T)}).
\]
\end{corollary}
This corollary is a consequence of Lemma \ref{AssumptionDecompoL1Norm}, \ref{LemmajumpHeight}, \ref{First111Lemma}, \ref{constantsConvRate}.
Next we state the main result of this work.
\begin{theorem}\label{BootstrapLemma}
The following convergence rates hold.
\begin{align}\label{optimalRate1}
\|\overline{u}-\overline{u}_\vartheta\|_{L^1(I)^m} =O(\tau^2 +h^2), &\quad	\vert c_i-c_{i,\vartheta} \vert =O(\tau^2 +h^2),\\
					\vert t_{j,i}-t_{j,\vartheta}^i \vert=O(\tau^2 +h^2), &\quad\vert c_{j}^i-c_{j,\vartheta}^i \vert=O(\tau^2 +h^2)
\end{align}
with $j=1,\cdots,m_i$, $i=1,\cdots,m$.
Furthermore, we have for the optimal states of $(\tilde{P})$ and $(\tilde{P}_{\vartheta}^{\text{semi}})$
\begin{equation}\label{optimalRate2}
\|\overline{y}-\overline{y}_{\vartheta}\|_{L^2(\Omega_T)} =O(\tau^2 +h^2).
\end{equation}
\end{theorem}
\begin{proof}
Using the inequality in (\ref{ineq1}), Corollary \ref{cor:Breg} and \ref{Cor221} and Corollary \ref{corRate}, we obtain for some $\epsilon>0$
\begin{multline*}
\|\overline{y}_{\vartheta}-\overline{y}\|_{L^2(\Omega_T)} \leq c\|\overline{y}-\hat{y}_{\vartheta}\|_{L^2(\Omega_T)}+c\|\overline{u}_{\vartheta}-\overline{u}\|_{L^1(I)^m}^{\frac{1}{2}}\|\overline{p} -\hat{p}_{\vartheta}\|_{C(I;L^2(\Omega))}^\frac{1}{2}\\
\leq c(\tau^2+h^2)+c\|\overline{u}_{\vartheta}-\overline{u}\|_{L^1(I)^m}^{\frac{1}{2}}\|\overline{p} -   \hat{p}_{\vartheta}\|_{C(I;L^2(\Omega))}^\frac{1}{2}\\
\leq c(\tau^2+h^2)+c\epsilon\|\overline{u}_{\vartheta}-\overline{u}\|_{L^1(I)^m} + \frac{c}{4 \epsilon} \|\overline{p} -\hat{p}_{\vartheta}\|_{C(I;L^2(\Omega))}\\
\leq c(\tau^2+h^2)+c\epsilon \|\overline{y}_{\vartheta}-\overline{y}\|_{L^2(\Omega_T)} + \frac{c}{4 \epsilon} \|\overline{p} - \hat{p}_{\vartheta}\|_{C(I;L^2(\Omega))}.
\end{multline*}
Consider a $\epsilon>0$ such that $c\epsilon=\frac{1}{2}$, then we have
\begin{equation*}
\|\overline{y}_{\vartheta}-\overline{y}\|_{L^2(\Omega_T)} \leq c(\tau^2+h^2+  \|\overline{p} -   \hat{p}_{\vartheta}\|_{C(I;L^2(\Omega))}).\quad
\end{equation*}
Then the a priori estimate (\ref{AprioriEsti}) implies
\begin{equation*}
\|\overline{p}-\hat{p}_{\vartheta}\|_{C(I;L^2(\Omega))}\leq
c(h^2 +\tau^2)\| \overline{y}-y_d\|_{C^1(I;H^1_0(\Omega))}.
\end{equation*}
 So we have  $\|\overline{y}_{\vartheta}-\overline{y}\|_{L^2(\Omega_T)}=O(\tau^2+h^2)$ and thus the same rate for the control in the $L^1(I)$-norm. Using the optimal rates of $\|\overline{y}_{\vartheta}-\overline{y}\|_{L^2(\Omega_T)}$ in (\ref{opti1}), (\ref{First111}), and (\ref{First112})
implies the optimal quadratic convergence rates for $c_{i,\vartheta}$, $t_{j,\vartheta}^i$ and $c_{j,\vartheta}^i$. 
\end{proof}
\begin{corollary}
For the BV representations of the optimal controls of $(\tilde{P})$ and $(\tilde{P}_{\vartheta}^{\text{semi}})$ hold
\begin{equation*}
\left\vert \|D_t\overline{u}\|_{M(I)}-\|D_t\overline{u}_{\vartheta}\|_{M(I)}\right\vert =O\left(
 \tau^2+h^2\right).
\end{equation*}
 Furthermore, $\overline{u}_{\vartheta}$ converges strictly in $BV(0,T)$ to $\overline{u}$ for $\vartheta \rightarrow 0$ with the convergence rate $O(\tau^2+h^2)$.
\end{corollary}
\begin{proof}
The statements are a consequence of (\ref{vffs}), (\ref{optimalRate2}), and Theorem \ref{BootstrapLemma}.
\end{proof}
\begin{remark}
	Based on the same techniques we have used so far, the following convergence rates can be shown for less regular data $(g_i)_{i=1}^m\subset L^\infty(\Omega)$ and $(y_d,y_0,y_1)\in C^1(\bar I;L^2(\Omega))\times H^1_0(\Omega) \times L^2(\Omega)$ as assumed in the problem $(P)$ at the beginning:
	\begin{itemize}
		\item[a)] $\|\bar y - \bar y_\vartheta\|_{L^2(\Omega_T)}= O(\tau^{2/3}+h^{2/3})$,
		\item[b)] $\vert J(\bar v,\bar c)-J_\vartheta(\bar v_\vartheta,\bar c_\vartheta)\vert = O(\tau^{2/3}+h^{2/3})$,
		\item[c)] $\|\bar u - \bar u_\vartheta\|_{L^1(I)^m}+\sum_{i=1}^m\left\vert\| D_t \bar u_i\|_{M(I)} - \|D_t \bar u_{\vartheta,i}\|_{M(I)} \right\vert= O(\tau^{2/3}+h^{2/3})$.
	\end{itemize}
In particular, statements of Theorem \ref{AprioriThm} has to be extended by \cite[Theorem 4.1., 4.3.]{[Zlot]} with respect to lower regular data chosen for $(g_i)_{i=1}^m$ and $(y_d,y_0,y_1)$.
\end{remark}
\section{Numerical Experiments}\label{sec:numerics}
In order to numerically verify the previously presented optimal error rates, an appropriate algorithm is of particular importance due to the variational discretization of problem $(P)$. Similarly as in \cite{[HMNV19]}, we solve the $BV$-control problem using Algorithm \ref{Algo1}, which is a modified version of the primal dual active point (PDAP) algorithm introduced in \cite[Algorithm 2]{PW19}. This method is based on a conditional gradient method, see \cite{BrediesPikkarainen:2013}. The algorithm calculates the derivative $\bar v\in M(I)^m$ and the offset $\bar c\in \mathbb R^m$ of the optimal control $\bar u$. The iterates for $v$ are given by a linear combination of Dirac measures $\left(\sum_{l=1}^{m_i}\lambda_{l,i}\delta_{t_{l,i}}\right)_{i=1}^m$. In every iteration the positions of the $m$ global maxima of
\[
p_{1,i}^k\colon t\mapsto \left\vert \int\nolimits_t^T \int\nolimits_{\Omega} L^*_{\vartheta}(S_\vartheta(v_k,c_k)-y_d)g_i ~dx~ds\right\vert,\quad i=1,\ldots,m
\]
are found. Then new Dirac measures are added at these positions. Finally a non-smooth optimization problem in terms of the magnitudes $\lambda$ and the constants $c$ is solved. The $L^1$-norm in the corresponding cost functional enhances sparsity in the vector of the magnitudes. If a magnitude is set to zero, the corresponding Dirac measure is erased from the current iterate.
For a convenient notation we define the map $\mathcal{U}_{\mathcal{A}}(\lambda):=\sum_{\ell\in \mathcal{A}}\lambda_\ell \delta_{\ell}$ for any finite set $\mathcal{A}\subset I$ and $\lambda\in \mathbb{R}^{\vert\mathcal{A} \vert}$. 
\begin{algorithm}
	Input: For $i=1,\cdots,m$ define $A_{0,i}\subset I$ with $\vert A_{0,i}\vert<\infty$, $\lambda_{0,i}\in \mathbb{R}^{\vert A_{0,i} \vert}$,
	$v_{0}:=(v_{0,i})_{i=1}^{m}:=\left(\mathcal{U}_{A_{0,i}}(\lambda_{0,i})\right)_{i=1}^{m}
	\in M(I)^m$, $c_0\in \mathbb{R}^m$, and $k=0$:\\
		Iterate:\\
		1. $\hat t_i=\argmax_{t\in\overline I}p_{1,i}^k=\left\vert \int\nolimits_t^T \int\nolimits_{\Omega} L^*_{\vartheta}(S_\vartheta(v_k,c_k)-y_d)g_i ~dx~ds\right\vert$ $\text{ for }i=1,\cdots,m$.\\
		2. Set $A_{k,i}:=\supp(v_{k,i})\cup\lbrace\hat t_i \rbrace $, for $i=1,\cdots,m$, $A_k=\bigcup_{i=1}^mA_{k,i}$ and compute
        \begin{equation*}
		(\overline \lambda,\overline c)=\argmin\limits_{
			\lambda\in \mathbb{R}^{\vert A_k\vert},~c\in \mathbb{R}^m}
		\frac{1}{2}\left\|S_{\vartheta}\left(
		\left(
		\mathcal{U}_{A_{k,i}}(\lambda_i)
		\right)_{i=1}^m,c\right)-y_d\right\|_{L^2(\Omega_T)}^2+\sum\nolimits_{i=1}^m\alpha_i \|\mathcal{U}_{A_{k,i}}(\lambda_{i})\|_{M(0,T)}.
		\end{equation*}
		\\
		3. Define $v_{k+1}=\left(\mathcal{U}_{A_{k,i}} (\overline \lambda_i) \right)_{i=1}^m$
		, $c_{k+1}=\overline c$; set $k=k+1$ and return to 1.
	\caption{BV-PDAP Algorithm}\label{Algo1}
\end{algorithm}
In the pure measure-valued case ($B(v,c)=v$) it is proven that this algorithm converges with a sublinear rate in terms of the cost functional. Under additional assumptions on the problem a linear rate is proven, see \cite{PW19}.
We consider a specific configuration of the input data for $(P)$, such that the solution is known explicitly and is a piecewise constant $BV-$function with finitely many jumps. We use a construction procedure for the solution which can be found in \cite{KuIch}. We fix the following scenario:
\begin{itemize}
\item $\Omega:=(-1,1)^2$ and $T=2$
\item $\alpha=6\cdot 10^{-3}$
\item $g(x):=\cos\left(\frac{\pi}{2} x_1\right)\cos\left(\frac{\pi}{2} x_2\right)$, thus $g\in \mathbb H^2$
\item $(y_0,y_1)=(0,0)$
\item $y_d=S(\bar v,\bar c)-(\partial_{tt}-\Delta)\varphi\in C^1(\bar I,H^1_0(\Omega))$ with
$\varphi(t,x):=\frac{3\pi\alpha}{2}\sin(2\pi t)\sin(\pi t)\cos\left(\frac{\pi}{2}x_1\right)\cos\left(\frac{\pi}{2}x_2\right)$.
\end{itemize}
Thus the data has the required regularity which we assumed in the proof of the quadratic convergence rates. Based on this input data the derivative $\bar v$ of the optimal control $\bar u$ has the form
\begin{equation}\label{Numbs1}
	\overline v=\delta_{1/3}-\delta_{1}+\delta_{5/3}\\
\end{equation}
and $\bar c=0$. Moreover this example fulfills Assumption \ref{as:jumps} since $p_{1,1}$ has the form
\(
p_{1,1}(t) =\alpha\sin\left(3\pi t/2\right)^3
\).
The function $y_d$ contains the optimal state $S(\bar v,\bar c)$, which cannot be obtained exactly. Thus we approximate it by its finite element solution on the finest grid level. More precisely, we set for the reference state $S(\bar v,\bar c)$ the mesh fineness in time and space to $\tau=2^{-9}$ and $h=2\cdot \sqrt{2}\cdot 2^{-8}$, respectively. All spatial integrals used for the numerical computation of $S(\bar v,\bar c)$ are calculated by a seven-point Gaussian quadrature rule different to the numerical calculations used for the approximated optimal state with respect to the BV-PDAP algorithm, where we considered a three-point Gauss quadrature rule. The function $p_{1,1}^k$ is computed exactly. Since they are piecewise quadratic in time, their global maxima can be attained at arbitrary points in $\bar I$, not necessarily at grid points. Thus $B(v^k,c^k)$ can have jumps outside of grid points. However, the integrals $\int_0^TB(v^k,c^k)e_i~dt$ with a hat function $e_i$ in the discrete state equation are calculated exactly. Moreover, the candidates for the global maxima of $p_{1,1}^k$ are also computed explicitly using the fact that $\partial_t p_{1,1}^k$ is piecewise linear. An additional $L^2$-regularization is added to the non-smooth optimization problem for the magnitudes $\lambda$ and the constants $c$. Then it is solved by a continuation strategy and a semi smooth Newton method, see also \cite{KuIch}.
\begin{figure}
	\begin{center}
		\includegraphics[width=13.2cm]{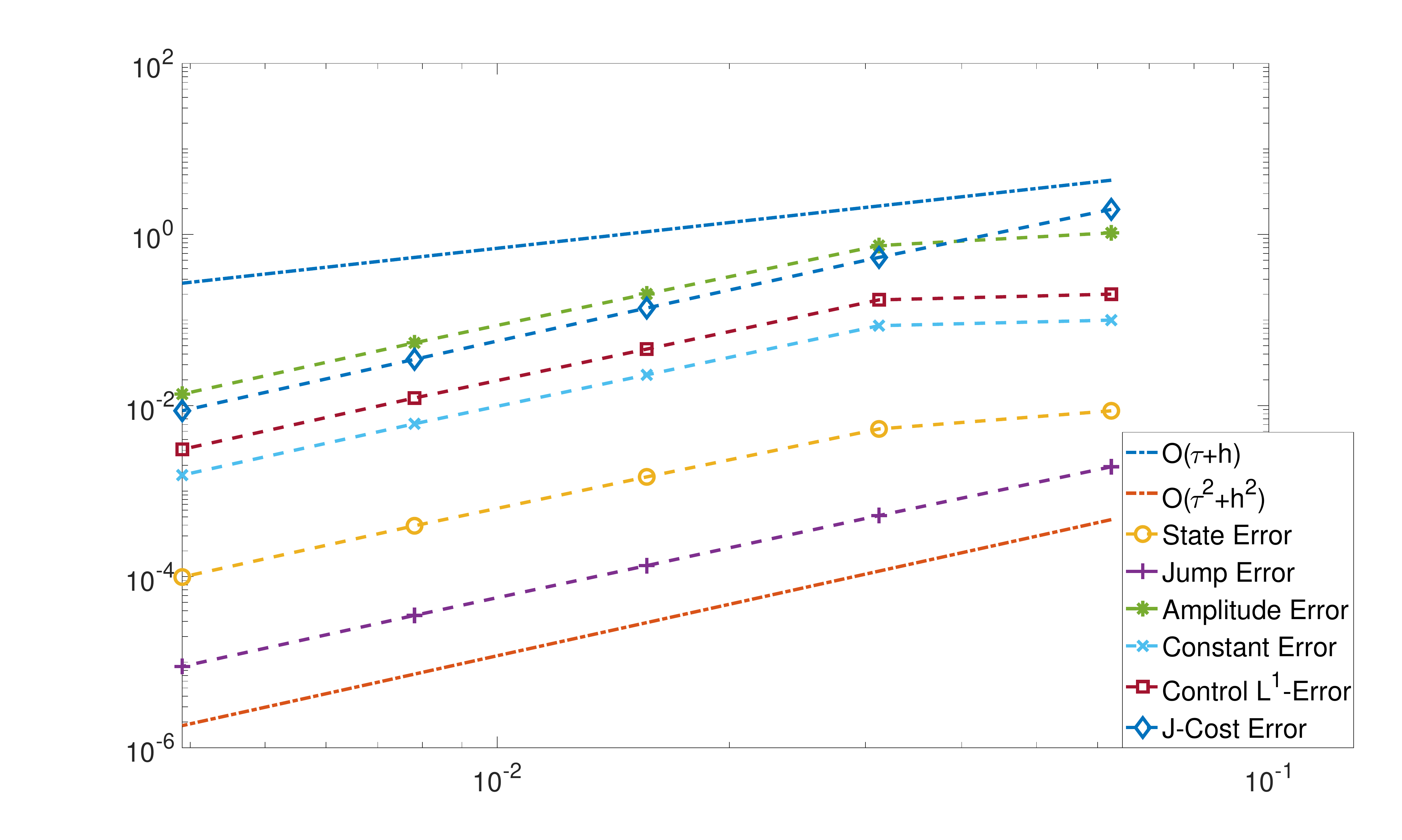}
	\end{center}
	\caption{In the legend we see the following error rates: "State Error" stands for the $\|S_{\tau,h}(\bar v_{\tau,h},\bar c_{\tau,h})-S(\bar v, \bar c)\|$ error. The "Jump Error", "Amplitude Error", and "Constant Error" are the error rates we defined in Theorem \ref{BootstrapLemma}. The line corresponding to "Control $L^1$-Error" is the $\|\bar u -\bar u_{\tau,h}\|_{L^1(I)}$ error and "J-Cost Error" stands for $\vert \tilde J(\bar v ,\bar c)-\tilde J(\bar v_{\tau,h},\bar c_{\tau,h})\vert$.	The discretization parameters are chosen as $\tau=2^{-k}$ and $h=2\cdot \sqrt{2} \cdot 2^{-k}$ with $k=4,5,\cdots,8$ and the grids are refined simultaneously.}\label{fig2}
\end{figure}
In Figure \ref{fig2} we observe that the error in state variable measured in the $L^2(\Omega_T)$-norm, in the control variable measured in the $L^1(I)$-norm as well as in the cost functional converge like $O(\tau^2+h^2)$. The same is true for the error in the jump positions $t_{j,1}$, the magnitudes of the jumps $c_j^1$ and the offset $c_1$. Thus the predicted error rates are verified.

\section*{Acknowledgments}
Sebastian Engel and Philip Trautmann were supported by the
International Research Training Group IGDK, funded by the German Science Foundation (DFG) and the Austrian Science Fund (FWF).

\pagebreak
\FloatBarrier



\end{document}